\documentclass[11pt]{article}

\date{}
\title{Implicit Representations and Factorial Properties of Graphs}

\author{A. Atminas\thanks{DIMAP and Mathematics Institute, University of Warwick, Coventry, CV4 7AL, UK.} \and 
A. Collins\footnotemark[1]  \and
V. Lozin\thanks{DIMAP and Mathematics Institute, University of Warwick, Coventry, CV4 7AL, UK. 
The author gratefully acknowledges support from DIMAP - the Center for Discrete Mathematics and its Applications 
at the University of Warwick, and from EPSRC, grant EP/I01795X/1. Email: V.Lozin@warwick.ac.uk} \and
V. Zamaraev\thanks{National Research University Higher School of Economics, Laboratory of Algorithms and Technologies for Network Analysis, Russia. 
Research of this author was supported by 
The National Research University Higher School of Economics
Academic Fund Program in 2014/2015 (research grant N. 14-01-0002); it was also partially supported by Russian Federation President Grant MK-1148.2013.1,
by RFFI grant N. 14-01-00515-a
and by Russian Federation Government grant N. 11.G34.31.0057.}}

\usepackage[cp1251]{inputenc}
\usepackage{amsmath, amsthm, amssymb}
\usepackage{colordvi}
\usepackage{graphicx}
\usepackage{tikz}
\usetikzlibrary{decorations.pathreplacing}

\tikzstyle{vertex}=[circle,fill=black!100,text=white,inner sep=0.8mm]
\tikzstyle{point}=[circle,fill=black,inner sep=0.1mm]

\begin{document}
\maketitle

\newtheorem{theorem}{Theorem}
\newtheorem{lemma}{Lemma}
\newtheorem{cor}{Corollary}
\newtheorem{definition}{Definition}
\newtheorem{remark}{Remark}
\newtheorem{conjecture}{Conjecture}

\begin{abstract}
The idea of implicit representation of graphs was introduced in 
[S. Kannan, M. Naor, S. Rudich, Implicit representation of graphs,
{\it SIAM J. Discrete Mathematics}, 5 (1992) 596--603] and can be 
defined as follows. A representation of an $n$-vertex graph $G$ 
is said to be implicit if it assigns to each vertex of $G$ 
a binary code of length $O(\log n)$ so that the adjacency of two vertices 
is a function of their codes. Since an implicit representation of an $n$-vertex 
graph uses $O(n\log n)$ bits, any class of graphs admitting such a representation
contains $2^{O(n\log n)}$ labelled graphs with $n$ vertices. In the terminology of 
[J. Balogh, B. Bollob\'{a}s, D. Weinreich, The speed of hereditary properties of graphs, 
{\it J. Combin. Theory} B 79 (2000) 131--156] such classes have at most factorial speed 
of growth. In this terminology, the implicit graph conjecture can be stated as follows: 
every class with at most factorial speed of growth which is  hereditary admits an implicit 
representation. The question of deciding whether a given hereditary class has at most 
factorial speed of growth is far from being trivial. In the present paper, we introduce 
a number of tools simplifying this question. Some of them can be used to obtain a stronger 
conclusion on the existence of an implicit representation. We apply our tools to reveal 
new hereditary classes with the factorial speed of growth. For many of them
we show the existence of an implicit representation. 
\end{abstract}


\section{Introduction}
We study simple graphs, i.e. undirected graphs without loops and multiple edges.
We denote by $M=M_G$ the adjacency matrix of a graph $G$ and by 
$m(u,v)=m_G(u,v)$ the element of $M$ corresponding to vertices $u$ and $v$,
i.e. $m(u,v)=1$ if $u$ and $v$ are adjacent and $m(u,v)=0$ otherwise.

Every simple graph on $n$ vertices can be represented by a binary word of length
$n\choose 2$ (half of the adjacency matrix), and if no a priory information about 
the graph is known, this representation is best possible in terms of its length.
However, if we know that our graph belongs to a particular class (possesses 
a particular property), this representation can be shortened. For instance, 
the Pr\"ufer code allows representing a labelled tree with $n$ vertices 
by a word of length $(n-2)\log n$ (in binary encoding)\footnote{All logarithms in this paper are of base 2}.
For labelled graphs, i.e. graphs with vertex set $\{1,2,\ldots,n\}$, we need $\log n$ bits for each vertex 
just to represent its label. That is why a representation of graphs from a specific class requiring $O(\log n)$ bits per vertex have been called in \cite{implicit} {\it implicit}. 

Throughout the paper by representing a graph we mean its coding, i.e. representing by a word in a finite alphabet
(in our case the alphabet is always binary). Moreover, we assume that different graphs are mapped to different words
(i.e. the mapping is injective) and that the graph can be restored from its code. For an implicit representation,
we additionally require that the code of the graph consists of the codes of its vertices, each of length $O(\log n)$,
and that the adjacency of two vertices, i.e. the element of the adjacency matrix corresponding to these vertices,
can be computed from their codes. 

Not every class of graphs admits an implicit representation, since a bound on the total length of the code
implies a bound on the number of graphs admitting such a representation. More precisely, only classes 
containing $2^{O(n\log n)}$ graphs with $n$ vertices can admit an implicit representation. 
However, this restriction does not guarantee that graphs in such classes can be represented implicitly. 
A simple counter-example can be found in \cite{Spinrad-book}. 
Even with further restriction to {\it hereditary classes}, i.e. those that are closed under taking induced subgraphs, 
the question is still not so easy. The authors of \cite{implicit}, who introduced the notion of an implicit 
representation, conjectured that every hereditary class with  $2^{O(n\log n)}$ graphs on $n$ vertices 
admits an implicit representation, and this conjecture is still open.  

In the terminology of \cite{SpHerProp}, hereditary classes containing $2^{O(n\log n)}$ labelled graphs 
on $n$ vertices are at most {\it factorial}, i.e. have at most factorial speed of growth. Classes with 
speeds lower than factorial are well studied and have a very simple structure. 
The family of factorial classes is substantially richer and the structure of classes in this family is more diverse.
It contains many classes of theoretical or practical importance, such as line graphs, interval graphs, 
permutation graphs, threshold graphs, forests, planar graphs and, even more generally, 
all proper minor-closed graph classes \cite{NSTW06}, all classes of graphs of bounded vertex degree, 
of bounded clique-width \cite{ALR09}, etc.    
   
In spite of the crucial importance of the family of factorial classes, except the definition very 
little can be said about this family in general, and the membership in this family is an open question 
for many particular graph classes. To simplify the study of this question, in the present paper we introduce 
a number of tools and apply them to reveal new members of this family. For some of them, we do even 
better and find an implicit representation. 

The organization of the paper is as follows. In the rest of this section, we introduce basic definitions
and notations related to the topic of the paper. In Section~\ref{sec:tools}, we define our tools and 
then in Section~\ref{sec:appl} we apply them to discover new factorial classes of graphs and new classes 
admitting an implicit representation.

\medskip
The vertex set and the edge set of a graph $G$ are denoted by $V(G)$ and $E(G)$, respectively.
Given a vertex $v\in V(G)$, we denote by $N(v)$ the neighbourhood of $v$, i.e. the set of 
vertices adjacent to $v$. 
For a subset $S \subset V(G)$, we denote by $N(S)$ the neighbourhood of $S$, i.e. the set of vertices outside $S$ that have at least one neighbour in $S$. 
The degree of $v$ is the number of its neighbours, i.e. $|N(v)|$,
and co-degree is the number of its non-neighbours, i.e. its degree in the complement of the graph.
As usual, we denote by $C_n$, $P_n$, $K_n$ the chordless cycle, the chordless path and the complete 
graph with $n$ vertices, respectively. 
By $K_{n,m}$ we denote a complete bipartite graph with parts of size $n$ and $m$.
Also, $O_n$ stands for the complement of $K_n$, i.e. 
the empty (edgeless) graph with $n$ vertices, and $S_{i,j,k}$ for the tree with three 
vertices of degree 1 being of distance $i,j,k$ from the only vertex of degree 3.    

In a graph, a {\it clique} is a set of pairwise adjacent vertices and an {\it independent set}
is a set of vertices no two of which are adjacent. A graph $G$ is {\it bipartite} if  $V(G)$ 
can be partitioned into at most two independent sets, and $G$ is a {\it split} graph if $V(G)$
can be partitioned into a clique and an independent set.

We say that a graph $H$ is an induced subgraph of a graph $G$ if $V(H)\subseteq V(G)$ 
and two vertices of $H$ are adjacent if and only if they are adjacent in $G$. 
If $G$ contains no induced subgraph isomorphic to $H$, we say that $G$ is $H$-free.  
Given a set $M$ of graphs, we denote by $Free(M)$ the class of graphs containing no induced subgraphs 
isomorphic to graphs in the set $M$. Clearly, for any set $M$, the class $Free(M)$ is {\it hereditary}, 
i.e. closed under taking induced subgraphs. The converse is also true: for any hereditary class $X$
there is a set $M$ such that $X=Free(M)$. Moreover, the minimal set $M$ with this property is unique.
We call $M$ the set of \textit{forbidden induced subgraphs} for the class $X$.

Given a class $X$, we write $X_n$ for the number of labelled graphs in $X$ and call $X_n$ the \textit{speed} of $X$.
The speed of hereditary classes (also known as hereditary properties\footnote{Throughout the paper, we use the two terms,
hereditary classes and hereditary properties, interchangeably.}) has been extensively studied in the literature.
In particular, paper \cite{Scheinerman} shows that the rates of the speed growth constitute discrete layers 
and distinguishes the first four of these layers: constant, polynomial, exponential and factorial.  
Independently, similar results have been obtained by Alekseev in \cite{Alekseev2}.
Moreover, Alekseev provided the first four layers with the description of all minimal classes,
i.e. he identified in each layer a family of classes every hereditary subclass of which belongs to a lower layer
(see also \cite{SpHerProp} for some more involved results). In particular, the factorial layer has 9 minimal classes, 
three of which are subclasses of bipartite graphs, three others are subclasses of co-bipartite graphs (complements of bipartite graphs) and 
the remaining three are subclasses of split graphs. The three minimal factorial classes of bipartite graphs are:
\begin{itemize}
\item $P^1=Free(K_3,K_{1,2})$, the class of graphs of vertex degree at most 1,
\item $P^2$, the class of ``bipartite complements" of graphs in $P^1$, i.e. the class of bipartite 
graphs in which every vertex has at most one non-neighbor in the opposite part, 
\item $P^3=Free(C_3,C_5,2K_2)$, the class of $2K_2$-free bipartite graphs, also known as chain 
graphs for the property that the neighborhoods of vertices in each part form a chain. 
\end{itemize}
The structure of graphs in these classes and in the related subclasses of split and co-bipartite graphs is very simple
and hence the problem of deciding whether a hereditary class has at least factorial speed of growth admits an easy 
solution. In the next section, we introduce a number of tools that can be helpful in deciding whether the speed of
a hereditary class is at most factorial.

\section{Tools}
\label{sec:tools}

\subsection{Modular decomposition}
Given a graph $G$ and a subset $U\subset V(G)$, we say that a vertex $x$ outside of $U$
{\it distinguishes} $U$ if it has both a neighbour and a non-neighbour in $U$. A proper subset 
of $V(G)$ is called a {\it module} if it is indistinguishable 
by the vertices outside of the set. A module is {\it trivial} if it consists of a single vertex. 
A graph every module of which is trivial is called {\it prime}. 

It is well-known (and not difficult to see) that a graph $G$ which is connected and co-connected
(the complement to a connected graph) admits a unique partition into maximal modules.
Moreover, for any two maximal modules $M_1$ and $M_2$, the graph $G$ contains either 
all possible edges between $M_1$ and $M_2$ or none of them. Therefore, by contracting 
each maximal module into a single vertex we obtain a graph which is prime (due to the maximality 
of the modules). This property allows a reduction of various graph problems from 
the set of all graphs in a hereditary class $X$ to prime graphs in $X$. In what follows, 
we show that the question of deciding whether a hereditary class is at most factorial 
also allows such a reduction. We start with the following technical lemma.

\begin{lemma}\label{lem:p7-1}
For any positive integers $k<n$, and $n_1,n_2,\dots,n_k$ such that $n_1+n_2+\dots +n_k=n$, 
the following inequality holds: $$k\log k+n_1\log n_1+n_2\log n_2+\ldots+n_k\log n_k\leq n\log n.$$
\end{lemma}

\begin{proof}
For $k=1$, the statement is trivial. Let $k>1$. The derivative of $f_a(x)=x \log x+(a-x) \log(a-x)$ is $\log x-\log (a-x)$, 
which is non-negative for $x \geq \frac{a}{2}$. In particular, this implies that for any two integers $m \geq n >1$ we have 
$f_{n+m}(m) \leq f_{n+m}(m+1)$. Hence,
\begin{equation}\label{log2:eq}
m \log m + n \log n \leq (m+1) \log (m+1) + (n-1) \log (n-1).
\end{equation}
Denote $n_0 = k$ and let $s$ be a number in $\{0,1,\ldots,k\}$ such that $n_s\ge n_i$ for all $i=0,1,\ldots,k$. 
			Applying inequality (\ref{log2:eq}) 
$(n_0-1) + \ldots + (n_{k}-1) - (n_s - 1) = n-n_s
$ 
times we obtain:

\medskip
\noindent
$n_0 \log n_0 + \ldots + n_k \log n_k \leq (n_s+n-n_s) \log (n_s + n - n_s) + 	1 \log 1 + \ldots + 1 \log 1 = n\log n.$
\end{proof}

\begin{theorem}\label{thm:modular}
Let $X$ be a hereditary class of graphs. If the number of prime $n$-vertex graphs in $X$ is  
$2^{O(n\log n)}$, then the number of all $n$-vertex graphs in $X$ is $2^{O(n\log n)}$.
\end{theorem}

\begin{proof}
For convenience, let us extend the notion of prime graphs by including in it all complete and all empty graphs. 
For each $n>2$, this extension adds to the set of prime graphs just two graphs, so we may assume that
the number of prime graphs in our class is at most $2^{cn\log n}$ for a constant $c>0$. 

For $n\ge 2$, let $f_n$ be an injection from the set of prime $n$-vertex graphs in $X$ to the binary sequences of length at most $cn\log n$. 
For each prime graph $P \in X$ on $n\ge 2$ vertices, let $f(P)=|n^{bin}|f_n(P)|$, where $n^{bin}$ is the binary expression of $n$. 
Thus, $f$ is an injection from the set of prime graphs in $X$ to the set of ternary words 
(i.e. words in the alphabet of three symbols $\{0,1,|\}$). For each $n$-vertex prime graph $P$ in $X$ the length of the word $f(P)$ 
is at most $cn\log n+\log n+3$. Observe that $cn\log n+\log n+3\leq (c+2)n\log n$ for $n\geq 2$. 
Therefore, each $n$-vertex prime graph in $X$ is represented by a ternary word of length at most $(c+2)n\log n$ for $n\geq2$. 
We claim that all the graphs in $X_n$ can be represented by different ternary words of length at most $(c+3)n\log n+n$.

Given a graph $G\in X_n$ we construct a modular decomposition tree $T$ of $G$ in which each node $x$  
corresponds to an induced subgraph of $G$, denoted $G_x$, and has a label, denoted $L_x$.
For the root, we define $G_x=G$. To define the children of $x$ and its label, we proceed as follows.
\begin{itemize}
\item Assume $G_x$ has at least two vertices, then 
\begin{itemize}
\item If $G_x$ is disconnected, we decompose it into connected components, associate each connected 
component with a child of $x$, and define $L_x=f(O_k)$, where $k$ is the number of connected components.
\item  If $G_x$ is the complement to a  disconnected graph, then we decompose it into co-components
(connected components of the complement), associate each co-component with a child of $x$, and 
define $L_x=f(K_k)$, where $k$ is the number of co-components.
\item If both $G_x$ and its complement are connected, then we decompose $G$ into maximal modules, 
associate each module with a child of $x$, and define $L_x=f(G_x^*)$, where $G_x^*$ is the prime graph obtained from $G_x$
by contracting each maximal module into a single vertex. 
\end{itemize}
\item Assume $G_x$ has just one vertex, and let $j\in \{1,2,\ldots,n\}$ be 
the label of that vertex in $G$. Then we define $x$ to be a leaf in $T$ and $L_x=j^{bin}$,
where $j^{bin}$ is the binary expression of $j$ of length $\log n$. 
\end{itemize}  

If $x$ is a non-leaf node of $T$, then it has $k\ge 2$ children, in which case its label has length at most $(c+2)k\log k$.
Otherwise $x$ is a leaf and its label has length $\log n$. 
Let $f(G)$ be the concatenation of the labels of all the nodes of $T$ in the order they appear in the depth-first search 
algorithm applied to $T$. Since the labels record the number of children for each node, it is not hard to see that 
we can reconstruct the original tree $T$ from the word $f(G)$, and hence we can reconstruct the graph $G$ from $f(G)$,
i.e. $f$ is an injection.


Let us prove that the length of the word $f(G)$ is at most $(c+3)n\log n+n$. 
The leaf nodes of $T$ contribute $n\log n$ bits to $f(G)$. Now by induction on $n\ge 2$ 
we show that the remaining nodes of $T$ contribute at most $(c+2)n\log n+n$ symbols to $f(G)$.
For $n=2$, this follows from the first part of the proof.
Now assume $n>2$. Let the root of the tree $T$ have $k$ children 
corresponding to induced subgraphs $G_1,\ldots,G_k$ of $G$ of  sizes $n_1,n_2,\dots n_k$ with $n_1+n_2+\dots +n_k=n$. 
Since $n_i<n$, by the induction hypothesis the internal nodes of $T_{G_i}$ contribute at most $(c+2)n_i\log n_i+n_i$
symbols to $f(G_i)$, where $T_{G_i}$ is the subtree of $T$ rooted at the vertex corresponding to subgraph $G_i$. 
Also, the label of the root has length at most $(c+2)k\log k$. 
Clearly the set of internal (non-leaf) nodes of $T$ coincides with the union of internal nodes of $T_{G_1} \ldots T_{G_k}$ and the root of $T$.
Hence, by Lemma~\ref{lem:p7-1}, the internal nodes of $T$ contribute at most $(c+2)n\log n+n$ symbols to $f(G)$.

Since we used 3 letters to represent graphs from $X$, the number of graphs in $X_n$ is 
at most $3^{(c+3)n\log n+n}\le 3^{(c+4)n\log n}=2^{c'n\log n}$, where $c'=(c+4)\log 3$,
i.e. $|X_n|=2^{O(n\log n)}$.
\end{proof}

\begin{cor}
If the set of prime graphs in a hereditary class $X$ belongs to a class which is at most factorial, 
then $X$ is at most factorial.
\end{cor}

\subsection{Functional vertices}
In this section, we introduce one more tool which is helpful in deciding 
whether a given class of graphs is factorial or not. Recall that by 
$m(x,y)$ we denote the element of the adjacency matrix corresponding 
to vertices $x$ and $y$.

\begin{definition}
For a graph $G=(V,E)$, we say that a vertex $y \in V$ is a function of a set of vertices $x_1, \ldots, x_k \in V$ 
if there exists a Boolean function $f: B^k\to B$ of $k$ variables such that for any vertex $z \in V \setminus \{y, x_1, \ldots, x_k\}$,
$$
m(y,z)=f(m(x_1,z),\ldots,m(x_k,z)).
$$ 
\end{definition}

\begin{theorem}\label{thm:function}
Let $X$ be a hereditary class of graphs and $c$ be a constant. If for every graph $G$ in $X$
there is a vertex $y$ and two disjoint sets $U$ and $R$ of at most $c$ vertices such that 
$y$ is a function of $U$ in the graph $G \setminus R$, then $|X_n|=2^{O(n\log n)}$.
\end{theorem}

\begin{proof} 
To prove the theorem, we will show by induction on $n$ that any $n$-vertex graph 
in this class can be described by $(2c+1)n\log n+ (2^c+2c)n$ bits.
This is clearly true for $n=1$ or $n=2$, so assume that every $(n-1)$-vertex graph 
in $X$ admits a description by a binary word of length at most $(2c+1)(n-1) \log(n-1) + (2^c+2c)(n-1)$.
Let $G$ be a graph in $X$ with $n$ vertices, $y$ a vertex in $G$ and $U= \{x_1, \ldots, x_k\}, R$ 
two sets as described in the statement of the theorem. For ease of notation we will call $y$ the functional vertex.

To obtain a description of $G$, we start by describing the label of $y$ by a binary word of length 
$\log n$. Next, we list each of the labels of the vertices in $R$, following each with a $0$ 
if $y$ is not adjacent to the vertex and a $1$ if $y$ is adjacent to the vertex. 
As there are at most $c$ vertices in $R$, this requires at most $c \log n+c$ bits. 
Next, we list each of the labels of the vertices in $U$, following each with a $0$ 
if $y$ is not adjacent to the vertex and a $1$ if $y$ is adjacent to the vertex. 
Similarly, this requires at most $c \log n+c$ bits. Then, as we know that $y$ is 
a function of the vertices in $U$ in the graph $G \setminus R$, there is a Boolean function $f$
that describes the adjacencies of the vertices in $G \setminus \{U \cup R \cup y\}$ to $y$. 
List the image of this function next. This requires at most $2^c$ bits, as there are at most $c$ vertices in $U$. 
Finally, append the description of the graph $G \setminus \{y\}$ which requires at most 
$(2c+1)(n-1)  \log(n-1)+ (2^c+2c)(n-1)$ bits by induction. So we have a description of $G$ 
by a binary word of length at most $(2c+1)  \log n+ (2^c+2c)+(2c+1)(n-1)  \log(n-1)+ (2^c+2c)(n-1)$ bits. 
Finally we see that
$$(2c+1) \log n+ (2^c+2c)+(2c+1)(n-1) \log (n-1)+ (2^c+2c)(n-1) \leq$$
$$(2c+1) \log n+ (2^c+2c)+(2c+1)(n-1) \log n+ (2^c+2c)(n-1) =$$
$$(2c+1)n \log n+ (2^c+2c)n$$ 
hence the result holds by induction.

For any two different vertices in $G$ this description can be used to identify if they are adjacent or not. 
If both vertices are different from $y$, their adjacency can be determined from the description of the graph $G \setminus \{y\}$.
Assume now that one of the vertices is $y$ and let $z$ be the other vertex. If $z\in R\cup U$, then the bit $m(y,z)$ is 
explicitly included in the description of $G$. If $z\in V(G) \setminus \{U \cup R \cup y\}$,
then $m(y,z)=f(m(x_1,z),\ldots,m(x_k,z))$. 
Note that the bits $m(x_1,z),\ldots,m(x_k,z)$ can be 
determined from the description of the graph $G \setminus \{y\}$, while the value of the function $f$ 
can be found in the description of $G$. This completes the proof of the theorem.
\end{proof}

A trivial example of a functional vertex is a vertex of bounded degree or co-degree, in which case
the set $U$ of variable vertices is empty. In this case, we can make a stronger conclusion.

\begin{lemma}\label{lem:degree}
Let $X$ be a hereditary class and $d$ a constant. If every graph in $X$ has a vertex
of degree or co-degree at most $d$, then $X$ admits an implicit representation.  
\end{lemma}

\begin{proof}
Since $X$ is hereditary, every graph $G$ in $X$ admits a linear order $P=(v_{i_1},\ldots,v_{i_n})$
of its vertices so that $v_{i_j}$ has degree or co-degree at most $d$ in the subgraph induced
by vertices $(v_{i_j},v_{i_{j+1}},\ldots,v_{i_n})$. Then an implicit representation for $G$ 
can be obtained by recording for each vertex $v$ its position in the linear order $P$ and
at most $d$ of its neighbours or non-neighbours among the vertices following $v$ in $P$.
One more bit is needed to indicate whether $v$ has at most $d$ neighbours or at most $d$ non-neighbours.
Clearly, this description completely defines the graph and hence provides an implicit 
representation for $G$. 
\end{proof}

The case when $y$ is a functional vertex
with $U=\{x\}$ and $f$ being a Boolean function of one variable mapping 0 to 0 and 1 to 1 
can be described as follows: $|N(x) \Delta N(y)| \leq c$, where $\Delta$ denotes the symmetric 
difference of two sets. This observation implies the following corollary which will be 
frequently used in the subsequent sections.

\begin{cor}\label{cor:Delta}
Let $X$ be a hereditary class of graphs and $c$ be a constant. 
If for every graph $G$ in $X$ there exist two vertices $x,y$ such that 
$|N(x) \Delta N(y)| \leq c$, then $X$ is at most factorial.
\end{cor}

\subsection{Covering of graphs}

\subsubsection{Locally bounded covering}

The idea of locally bounded coverings was introduced in \cite{coverings} to study factorial 
properties of graphs. This idea can be described as follows.
 
Let $G$ be a graph. A set of graphs $H_1,\dots ,H_k$ is called a \textit{covering} of $G$ if 
the union of $H_1,\dots ,H_k$ coincides with $G$, i.e. if $V(G)=\bigcup\limits_{i=1}^k V(H_i)$ and 
$E(G)=\bigcup\limits_{i=1}^k E(H_i)$. The following result was proved in \cite{coverings}.

\begin{lemma}\label{lem:covering}
Let $X$ be a class of graphs and $c$ a constant. 
If every graph $G\in X$ can be covered by graphs from a class $Y$ with $\log Y_n=O(n\log n)$ 
in such a way that every vertex of $G$ is covered by at most $c$ graphs, then $\log X_n=O(n\log n)$.
\end{lemma}

Now we derive a similar result for implicit representations of graphs.

\begin{lemma}\label{lem:implicit}
Let $X$ be a class of graphs and $c$ a constant. 
If every graph $G\in X$ can be covered by graphs from a class $Y$ admitting an implicit representation  
in such a way that every vertex of $G$ is covered by at most $c$ graphs, then graphs in $X$ also 
admit an implicit representation.
\end{lemma}

\begin{proof} 
Let $H_1,\ldots,H_k\in  Y$ be a covering of a graph $G\in X$ such that 
every vertex of $G$ is covered by at most $c$ graphs, where $c$ is a constant independent of $G$. 
Denote $n_i=|V(H_i)|$, $n=|V(G)|$. Then 
\begin{equation}\label{eq:1}
k\le \sum\limits_{i=1}^k n_i\le cn. 
\end{equation}

Let $\phi_i$ be an implicit representation of $H_i$, i.e. a binary word containing
for each vertex of $H_i$ a code of length $O(\log n_i)$ so that the adjacency of two vertices 
can be computed from their codes.
 
Now we construct an implicit representation of $G$ as follows. To each 
vertex $j\in V(G)$ we assign a binary word $\psi_j$ containing for each graph 
$H_i$ covering $j$ the index $i$ and the code of vertex $j$ in the representation $\phi_i$ of $H_i$.
Clearly, the adjacency of two vertices $j, k\in V(G)$ can be determined from their codes 
$\psi_j$ and $\psi_k$, because they are adjacent in $G$ if and only if there is a graph $H_i$
which covers both of them and in which these vertices are adjacent. 
Since each vertex $j\in V(G)$ is covered by at most $c$ graphs, the length of $\psi_j$ is 
at most $c\log k +\sum\limits_{i=1}^{c}O(\log n_{j_i}) = c(\log k +O(\log n))$. Together with (\ref{eq:1}) 
this implies that  $|\psi_j|=O(\log n)$
and hence $\{\psi_j\ :\ j=1,\ldots,n\}$ is an implicit representation of $G$.
\end{proof}

This result allows to conclude the existence of implicit representations for a variety of graph
properties. For instance it is known that forests admit an implicit representation (also follows
from Lemma~\ref{lem:degree} as well as from Lemma~\ref{lem:implicit}). The example of forests leads to many more 
important conclusions with the help of the notion of arboricity. The arboricity of a graph is 
the minimum number of forests into which its edges can be partitioned. Many classes of theoretical 
or practical importance have bounded arboricity and hence, by Lemma~\ref{lem:implicit}, admit implicit representations, which is 
the case for graphs of bounded vertex degree, of bounded genus, of bounded thickness and for all 
proper minor-closed graph classes.

\subsubsection{Partial covering}

One more tool was introduced in \cite{chordal-bipartite} and can be stated as follows. 

\begin{lemma}\label{lem:partial}
Let $X$ be a hereditary class. If there is a constant $d \in \mathbb{N}$ and  
a hereditary class $Y$ with at most factorial speed of growth such that 
every graph $G=(V,E) \in X$ contains a non-empty subset $A\subseteq V$ such that $G[A] \in Y$ and
each vertex $a \in A$ has either at most $d$ neighbours or at most $d$ non-neighbours in $V-A$,
then $X$ is at most factorial.
\end{lemma}

Now we derive a similar result for implicit representations. This result can be viewed 
as a generalization of Lemma~\ref{lem:degree}.

\begin{lemma}\label{lem:partial-implicit}
Let $X$ be a hereditary class. If there is a constant $d \in \mathbb{N}$ and  
a hereditary class $Y$ which admits an implicit representation such that 
every graph $G=(V,E) \in X$ contains a non-empty subset $A\subseteq V$ such that
$G[A] \in Y$ and each vertex $a \in A$ has either at most $d$ neighbours or at most $d$ non-neighbours in $V-A$,
then $X$ admits an implicit representation.
\end{lemma}

\begin{proof}
First, we represent $G[A]$ implicitly  (which is possible, because $G[A] \in Y$ and $Y$ admits an implicit representation)
and then add to the code of each vertex $v$ of $G[A]$ the list of at most $d$ neighbours or non-neighbours of $v$ in the 
rest of the graph. This describes $G[A]$ and its adjacency to the rest of the graph implicitly, 
i.e. with $O(\log n)$ bits per each vertex of $A$. Then the set $A$ can be deleted (or simply ignored) 
and the procedure can be applied to the rest of the graph, which is possible because $X$ is a hereditary class. 
Eventually, we obtain an ordered sequence of sets $A_0=A, A_1,A_2,\ldots,A_k$ ($k\le n$) such that for each $i\ge 0$, the graph $G[A_i]$ 
and its adjacency to the vertices in $A_{i+1},\ldots, A_{k}$ are described implicitly. To complete the description
of $G$, we assign to each vertex $v\in V(G)$ the index the set $A_i$ it belongs to. Now the adjacency of two vertices 
$u, v\in V(G)$ can be tested as follows: if both of them belong to the same set $A_i$, then their adjacency 
can be determined through their codes in the implicit representation of $G[A_i]$, and if $u\in A_i$ and $v\in A_j$
with $i<j$, then their adjacency can be determined by looking at the list of neighbours (or non-neighbours) of $u$ 
which is stored in the label of $u$.  
\end{proof}

\subsection{Remarks}
\label{sec:remarks}

In Theorem~\ref{thm:function}, Lemma~\ref{lem:degree}, Corollary~\ref{cor:Delta}, Lemmas~\ref{lem:partial} and~\ref{lem:partial-implicit},
to prove the corresponding statements for a class $X$, we require that {\it every} graph in $X$ has a subset of vertices (or a single vertex)
satisfying certain properties. This requirement can be relaxed if some graphs in $X$ belong to a class $Z$ that satisfy conditions 
of the corresponding statement. In this case, the existence of a subset (or a vertex) with a particular property can be required only 
for graphs in $X-Z$. For easy reading, we do not introduce this relaxation into the text of the corresponding results. But we keep it in mind
when we apply these results in the next section.


\section{Applications}
\label{sec:appl}

In this section, we apply the tools developed in the previous one in order 
to reveal new factorial classes of graphs. In some cases, we show that these 
classes admit an implicit representation. 
To simplify the study of factorial graph properties, in \cite{LMZ11} the following conjecture was proposed. 

\medskip
{\bf Conjecture on factorial properties. \it A hereditary graph property $X$ is factorial if and only if the fastest of the 
following three properties is factorial: bipartite graphs in $X$, co-bipartite graphs in $X$,
split graphs in $X$}.

\medskip
To justify this conjecture we observe that if in the text of the conjecture we replace 
the word ``factorial'' by any of the lower layers (constant, polynomial or exponential),
then the text becomes a valid statement. Also, the ``only if'' part of the conjecture is 
true, because all minimal factorial classes are subclasses of bipartite, co-bipartite 
or split graphs. Also, in \cite{LMZ11} this conjecture was verified for all hereditary 
classes defined by forbidden induced subgraphs with at most 4 vertices.   

The above conjecture reduces the question of deciding the membership in the factorial layer
from the family of all hereditary properties to those which are bipartite, co-bipartite and split. 
Taking into account the obvious relationship between bipartite, co-bipartite and split graphs,
this question can be further reduced to hereditary properties of bipartite graphs only.

When we talk about bipartite graphs, we assume that each graph is given together with 
a bipartition of its vertex set into two parts (independent sets), say top and bottom,
and we denote a bipartite graph with parts $A$ and $B$ by $G=(A,B,E)$, where $E$, as before, 
stands for the set of edges. The {\it bipartite complement} of a bipartite graph $G=(A,B,E)$
is the bipartite graph $\widetilde{G}=(A,B,E')$, where two vertices $a\in A$ and $b\in B$
are adjacent in $G$ if and only if they are not adjacent in $\widetilde{G}$. By $O_{n,m}$
we denote the bipartite complement of $K_{n,m}$.

For connected graphs, the bipartition into two independent sets is unique (up to symmetry). 
A disconnected bipartite graph can admit several different bipartitions, and this distinction 
between different bipartitions can be crucial if our graph is {\it forbidden}. Consider, for instance,
the graph $2K_{1,2}$ (a disjoint union of two copies of $K_{1,2}$). Up to symmetry, 
it admits two different bipartitions and by forbidding one of them we obtain a subclass 
of bipartite graphs which is factorial, while by forbidding the other we obtain a superfactorial
subclass of bipartite graphs. This is because the {\it bipartite complement} of one of them
does not contain any cycle, while the {\it bipartite complement} of the other contains 
a $C_4$. More generally, in \cite{Allen} the following result was proved.

\begin{theorem}\label{thm:Allen}
Let $G$ be a bipartite graph. If either $G$ or its bipartite complement contains a cycle,
then the class of $G$-free bipartite graphs is superfactorial. If both $G$ and its bipartite 
complement are acyclic and $G$ is different from $P_7$, then the class of $G$-free bipartite graphs is at most factorial. 
\end{theorem}

Moreover, for most bipartite graphs $G$ such that neither $G$ nor its bipartite complement contains a cycle,
paper \cite{Allen} proves a stronger result. To state this result, let us observe that when we say that 
a bipartite graph $G$ contains a bipartite graph $H$ as an induced subgraph, we do not specify which part 
of $H$ is mapped to which part of $G$. However, sometimes this specification is important and if all induced 
copies of $H$ appear in $G$ with all bottom parts of $H$ being in the same part of $G$, then we say that 
$H$ is contained in $G$ {\it one-sidedly}. If at least one of the two possible appearances of $H$ is missing in $G$, 
we say that $G$ contains no {\it one-sided} copy of $H$. 

\begin{theorem}\label{thm:Allen+}\cite{Allen}
If both $G$ and its bipartite complement are acyclic and $G$ is different from 
$P_7$, $S_{1,2,3}$, $S_{1,2,2}$ and from the bipartite complement of $S_{1,2,2}$, 
then the class of bipartite graphs containing no one-sided copy of $G$ is at most factorial. 
\end{theorem} 

According to Theorem~\ref{thm:Allen}, the class of $P_7$-free bipartite graphs is the only subclass 
of bipartite graphs defined by a single forbidden induced subgraph for which the membership in the factorial 
layer is an open question. 

To better understand the structure of $P_7$-free bipartite graphs, in the present paper
we study subclasses of this class defined by one additional forbidden induced subgraph
and prove, in particular, that for every graph $G$ with at most 6 vertices the class of 
$(P_7,G)$-free bipartite graphs is at most factorial.
 
Many of our results can be extended, with no extra work, to the more general case of 
bipartite graphs of bounded chordality, i.e. $(C_{k},C_{k+1},\ldots)$-free bipartite graphs
for a constant $k$ (the chordality of a graph is the length of a longest chordless cycle).
For $k=4$, the class of $(C_{k},C_{k+1},\ldots)$-free bipartite graphs coincides with forests and this 
class is factorial. However, for any $k>4$, the class of $(C_{k},C_{k+1},\ldots)$-free bipartite graphs 
is superfactorial. Indeed, each of these classes contain the class of $(C_6,C_8,\ldots)$-free bipartite graphs, also known as 
chordal bipartite graphs, which is a superfactorial class, as the number of $n$-vertex labelled graphs in 
this class is $2^{\Theta(n\log^{2}{n})}$ \cite{Spinrad}. Moreover, the class of chordal bipartite graphs
is not a minimal superfactorial class, which is due to the following result proved in \cite{chordal-bipartite},
where $2C_4$ denotes the disjoint union of 2 copies of $C_4$, and $2C_4+e$ is the graph obtained 
from $2C_4$ by adding one edge between the two copies of $C_4$.

\begin{lemma}
The class of $(2C_4, 2C_4+e)$-free chordal bipartite graphs is superfactorial.
\end{lemma}

On the other hand, most of the hereditary subclasses of chordal bipartite graphs 
studied in the literature, such as forests, bipartite permutation, convex graphs, 
are factorial. Also, several results on  factorial properties of chordal bipartite graphs 
were obtained in \cite{chordal-bipartite} and \cite{forest-chordal-bipartite}.
In particular, in \cite{forest-chordal-bipartite} the following result was proved. 

\begin{lemma}\label{lem:forest}
For any forest $F$, the class of $F$-free chordal bipartite graphs is at most factorial.
\end{lemma}

This result cannot be extended to $(C_{k},C_{k+1},\ldots)$-free bipartite graphs
for $k>6$, because, for instance, $(C_{10},C_{11},\ldots)$-free bipartite graphs
contain all $P_{8}$-free bipartite graphs, which is a superfactorial class 
(as the bipartite complement of $P_8$ contains a $C_4$),
and $(C_{8},C_{10},\ldots)$-free bipartite graphs contain $P_{7}$-free bipartite graphs, 
a class for which the membership in the factorial layer is an open question.

However, for some graphs $G$ containing a cycle, it is possible to prove the membership
of $(G,C_{k},C_{k+1},\ldots)$-free bipartite graphs in the factorial layer for any value 
of $k$. For $k=6$ (i.e. for chordal bipartite graphs) several results of this type 
have been obtained in \cite{chordal-bipartite}. In Section~\ref{sec:small-chordality},
we extend these results to bipartite graphs of chordality at most $k$ for arbitrary value of $k$. 
We also obtain a number of new results for such classes. 

\medskip
In Section~\ref{sec:P7}, we restrict ourselves further and consider subclasses of $P_7$-free
bipartite graphs, which is a special case of $(C_{8},C_{10},\ldots)$-free bipartite graphs.
We systematically study subclasses of $P_7$-free bipartite graphs defined by one additional 
forbidden induced subgraph and show that for every graph $G$ with at most 6 vertices the class of 
$(P_7,G)$-free bipartite graphs is at most factorial.


\subsection{Bipartite graphs of small chordality}
\label{sec:small-chordality}


In this section, we study $(C_{k},C_{k+1},\ldots)$-free bipartite graphs.
For $k=6$, this class is known as chordal bipartite graphs and is known to be 
superfactorial \cite{Spinrad}. Therefore, bipartite graphs of chordality at most $k$ constitute a superfactorial class for all $k\ge 6$.
Various factorial properties of chordal bipartite graphs were studied in \cite{chordal-bipartite}.
In the present section, we generalize most of them to arbitrary values of $k$ and obtain 
a number of new results for such classes. We start with the following general result.

\begin{lemma}\label{lem:KSS-free}
For any natural numbers $p\ge 2$ and $k\ge 6$, the class of $(K_{p,p},C_{k},C_{k+1},\ldots)$-free bipartite graphs
admits an implicit representation and hence is at most factorial.
\end{lemma} 

\begin{proof}
In \cite{KSS}, it was shown that for every graph $H$ and for every natural $p$, there exists $d=d(H,p)$
such that every graph of average degree at least $d$ contains either a $K_{p,p}$ as a (not necessarily induced)
subgraph or an induced subdivision of $H$. This implies that every 
$(K_{p,p},C_{k},C_{k+1},\ldots)$-free bipartite graph $G$ contains a vertex of degree less than $d(C_{k},p)$,
since otherwise the average degree of $G$ is at least $d(C_k,p)$, in which case it must contain either 
an induced subdivision of $C_{k}$ (which is forbidden) or a $K_{p,p}$ as a subgraph (which is also forbidden,
else an induced copy of $K_{p,p}$ or $K_3$ arises). This implies, by Lemma~\ref{lem:degree}, that 
the class of $(K_{p,p},C_{k},C_{k+1},\ldots)$-free bipartite graphs
admits an implicit representation and hence is at most factorial.
\end{proof}

For $k=6$, i.e. for chordal bipartite graphs, the result of Lemma~\ref{lem:KSS-free} was derived,
by different arguments, in \cite{chordal-bipartite}. In particular, in that paper it was proved 
that $K_{p,p}$-free chordal bipartite graphs have bounded tree-width. This is a stronger conclusion
and we believe that the same conclusion holds for $K_{p,p}$-free bipartite graphs of chordality at 
most $k$ for each value of $k$. More generally we conjecture:

\medskip
{\bf Conjecture}. {\it For all $r$, $p$ and $k$, there is a $t=t(r,p,k)$ such that any $(K_r,K_{p,p})$-free graph 
of chordality at most $k$ has tree-width at most $t$.}

\medskip
We leave this conjecture for future research. In the present paper, we extend the result of Lemma~\ref{lem:KSS-free}
in a different way. In \cite{chordal-bipartite}, it was proved that the class of chordal bipartite graphs containing 
no induced $K_{p,p}+K_1$ is at most factorial by showing that every $K_{p,p}+K_1$-free bipartite graph containing 
a $K_{s,s}$ with $s=p(2^{p-1}+1)$ contains a vertex which has at most $2p-2$ non-neighbours in the opposite part.
Together with Lemma~\ref{lem:degree}, this immediately implies the following extension of Lemma~\ref{lem:KSS-free}.
   
\begin{lemma}\label{lem:KSS-free+}
For any natural $p\ge 2$ and $k\ge 6$, the class of $(K_{p,p}+K_1,C_{k},C_{k+1},\ldots)$-free bipartite graphs
admits an implicit representation and hence is at most factorial.
\end{lemma}    

Below we further extend this result and obtain a number of other results for 
subclasses of bipartite graphs of bounded chordality.

\subsubsection{$Q(p)$-free bipartite graphs of bounded chordality}

We denote by $Q(p)$ the graph obtained from $K_{p,p}+K_1$ by adding 
a new vertex to the smaller part of the graph and connecting it 
to every vertex in the opposite side. The graph $Q(2)$ is represented 
in Figure~\ref{fig:C}. 

\begin{figure}[ht]
\begin{center}
	\begin{tikzpicture}
  		[scale=.6,auto=left]

		\node[vertex] (x1) at (0,0)   { };
		\node[vertex] (x2) at (1,0)   { };
		\node[vertex] (x3) at (2,0)   { };
		\node[vertex] (x4) at (0,2)   { };
		\node[vertex] (x5) at (1,2)   { };
		\node[vertex] (x6) at (2,2)   { };

		\foreach \from/\to in {x1/x4,x1/x5,x2/x4,x2/x5,x1/x6,x6/x2,x6/x3}
    	\draw (\from) -- (\to);
	\end{tikzpicture}
\end{center}
\caption{Graph $Q(2)$}
\label{fig:C}
\end{figure}
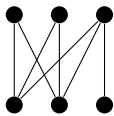

\begin{theorem}\label{thm:Qp}
For any natural $k$ and $p$, the class of $Q(p)$-free bipartite graphs of chordality at most $k$ 
admits an implicit representation and hence is at most factorial. 
\end{theorem}

\begin{proof}
Let $G$ be a $Q(p)$-free bipartite graph of chordality at most $k$. If $G$ contains no $K_{p^2,p^2}$,
it admits an implicit representation by Lemma~\ref{lem:KSS-free}. Therefore, we assume
that $G$ contains a $K_{p^2,p^2}$. Moreover, by Lemma~\ref{lem:implicit} we may assume that $G$ is connected.

We denote the two parts in the bipartition of $G$ by $A$ and $B$ and extend the $K_{p^2,p^2}$ contained in $G$ 
to a maximal (with respect to set inclusion) complete bipartite graph $H$ with parts $A_0 \subseteq A$ and $B_0 \subseteq B$. 
The set $A-A_0$ can further be split into the set $A_1$ of vertices that have neighbours in $B_0$ 
and the set $A_2$ of vertices that have no neighbours in $B_0$. Observe that due to the maximality of $H$
each vertex of $A_1$ has at least one non-neighbour in $B_0$. We further split $A_1$ into the 
set $A_1'$ of vertices with at most $p-1$ non-neighbours in $B_0$ and   
the set $A_1''$ of vertices with at least $p$ non-neighbours in $B_0$.
The set $B-B_0$ can be split into $B_1'$, $B_1''$ and $B_2$ analogously. 
We claim that

\begin{itemize}
\item[(1)]{ \it $A_1''=B_1''=\emptyset$.} Suppose this is not true and let $x$ be a vertex in $A_1''$
(without loss of generality). By definition $x$ must have a neighbour $y$ and $p$ non-neighbours in $B_0$.
Then these vertices together with any $p$ vertices in $A_0$ induce a $Q(p)$. 

\item[(2)]{ \it $A_2=B_2=\emptyset$.} Suppose to the contrary that $A_2$ contains a vertex $x$. 
Then because of Claim (1) and  due to the connectedness of $G$, vertex $x$ must have a neighbour $y\in B_1'$.
Since $y$ has at most $p-1$ non-neighbours in $A_0$, it has at least $p$ neighbours in $A_0$.
Then these $p$ neighbours together with $x,y$ and any $p$ vertices in $B_0$ induce a $Q(p)$. 

\item[(3)]{\it The subgraph of $G$ induced by $A_1'\cup B_1'$ is $K_{p,p}+K_1$-free}. 
Assume $G[A_1'\cup B_1']$ contains an induced $K_{p,p}+K_1$ and let, without loss of generality, the $p+1$ vertices of 
this graph belong to $A_1'$. Each of these $p+1$ vertices have at most $p-1$ non-neighbours in 
$B_0$ and since $B_0$ contains at least $p^2$ vertices we conclude that there must be a vertex 
in $B_0$ adjacent to each of the $p+1$ vertices of the copy of $K_{p,p}+K_1$. But then together (that 
vertex and the copy of $K_{p,p}+K_1$) induce a $Q(p)$ in $G$.  
\end{itemize}

Claim (3) implies by Lemma~\ref{lem:KSS-free+} that $G[A_1'\cup B_1']$ admits an implicit representation.
Besides, every vertex of $A_1'\cup B_1'$ has at most $p-1$ non-neighbours in the rest of the graph. 
Therefore, by Lemma~\ref{lem:partial-implicit} (as well as by Lemma~\ref{lem:implicit}) we conclude
that $G$ admits an implicit representation. 
\end{proof}

\subsubsection{$L(s,p)+O_{0,1}$-free bipartite graphs of bounded chordality}

By $L(s,p)$ we denote a bipartite graph obtained from $K_{2,p}$ by 
adding $s$ pendant edges to one of the vertices of degree $p$. By adding 
an isolated vertex to the bottom part of the graph, we obtain  $L(s,p)+O_{0,1}$
(see example of $L(2,2)+O_{0,1}$ in Figure~\ref{fig:L(2,2)}).

\begin{figure}[ht]
\begin{center}
	\begin{tikzpicture}
  		[scale=.6,auto=left]


		\node[vertex] (x1) at (0,0)   { };
		\node[vertex] (x2) at (1,0)   { };
		\node[vertex] (x3) at (2,0)   { };
		\node[vertex] (x4) at (3,0)   { };
		\node[vertex] (x7) at (4,0)   { };
		\node[vertex] (x5) at (2,2)   { };
		\node[vertex] (x6) at (3,2)   { };

		\foreach \from/\to in {x1/x5,x2/x5,x3/x5,x4/x5,x3/x6,x4/x6}
    	\draw (\from) -- (\to);
	\end{tikzpicture}
\end{center}
\caption{Graph $L(2,2)+O_{0,1}$}
\label{fig:L(2,2)}
\end{figure}
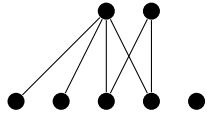

\begin{theorem}\label{thm:L}
For any natural $k,s,p$, the class of $L(s,p)+O_{0,1}$-free bipartite graphs of chordality at most $k$ 
is at most factorial.
\end{theorem}

\begin{proof}
Let $G$ be an $L(s,p)+O_{0,1}$-free bipartite graph of chordality at most $k$. If $G$ contains no $K_{2,p}+O_{0,1}$, then it contains no $K_{t,t}+K_1$, where
$t = \max\{ 2, p \}$, and hence
it admits an implicit representation by Lemma~\ref{lem:KSS-free+}. Therefore, we assume
that $G$ contains an induced copy of $K_{2,p}+O_{0,1}$ and let $x,y$ be the two vertices of degree $p$ in that copy. 
Vertex $x$ cannot have $s$ or more private neighbours (i.e. neighbours which are not adjacent to $y$),
since otherwise any $s$ of these neighbours together with the $K_{2,p}+O_{0,1}$ would induce an $L(s,p)+O_{0,1}$. 
The analogous statement also holds for $y$. Therefore, $|N(x) \Delta N(y)| \leq 2(s-1)$ and hence,
by Corollary~\ref{cor:Delta}, the class of $L(s,p)+O_{0,1}$-free bipartite graphs of chordality at most $k$ 
is at most factorial. 
\end{proof}

\medskip
To conclude this section, we observe that the result of Theorem~\ref{thm:L} is best possible in the sense 
that by increasing either of the indices of the second term in the definition of the forbidden graph we obtain 
a superfactorial class. More precisely:

\begin{remark}\label{rem:1}
For any $s,p\ge 1$ and $k\ge 8$, the classes of $L(s,p)+O_{1,1}$-free, $L(s,p)+O_{0,2}$-free and $L(s,p)+O_{2,0}$-free 
bipartite graphs of chordality at most $k$ are superfactorial.
\end{remark}

This conclusion follows from the fact that $L(s,p)+O_{1,1}$ , $L(s,p)+O_{0,2}$ and $L(s,p)+O_{2,0}$, as well as all bipartite cycles of length 
more than 8 contain $\widetilde{C}_4$, and hence the corresponding classes contain all $\widetilde{C}_4$-free bipartite 
graphs, which form a superfactorial class by Theorem~\ref{thm:Allen}.

\subsubsection{$M(p)$-free bipartite graphs of bounded chordality}

By $M(p)$ we denote the graph obtained from $L(1,p)$ by adding one vertex 
which is adjacent only to the vertex of degree 1 in the $L(1,p)$.  Figure~\ref{fig:B}
represents the graph $M(3)$. 

\begin{figure}[ht]
\begin{center}
	\begin{tikzpicture}
  		[scale=.6,auto=left]

		\node[vertex] (x0) at (-1,0)  { };
		\node[vertex] (x1) at (0,0)   { };
		\node[vertex] (x2) at (1,0)   { };
		\node[vertex] (x3) at (2,0)   { };
		\node[vertex] (x4) at (0,2)   { };
		\node[vertex] (x5) at (1,2)   { };
		\node[vertex] (x6) at (2,2)   { };

		\foreach \from/\to in {x1/x4,x1/x5,x2/x4,x2/x5,x5/x3,x6/x3,x0/x4,x0/x5}
    	\draw (\from) -- (\to);
	\end{tikzpicture}
\end{center}
\caption{Graph $M(3)$}
\label{fig:B}
\end{figure}
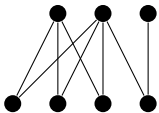

\begin{theorem}
For any natural $k$ and $p$, the class of $M(p)$-free bipartite graphs of chordality at most $k$ 
is at most factorial.
\end{theorem}

\begin{proof}
Let $G$ be an $M(p)$-free bipartite graph of chordality at most $k$. If $G$ contains no $K_{p^2,p^2}$,
it admits an implicit representation by Lemma~\ref{lem:KSS-free}. Therefore, we assume
that $G$ contains a $K_{p^2,p^2}$. Moreover, by Theorem~\ref{thm:modular} we may assume that $G$ is prime 
and hence is connected.
We denote the two parts in the bipartition of $G$ by $A$ and $B$ and split them into 
$A_0,A_1',A_1'',A_2$ and $B_0,B_1',B_1'',B_2$ as in Theorem~\ref{thm:Qp}.

Let $M^*(p)$ denote the subgraph of $M(p)$ obtained by deleting the vertex of degree $p+1$
(i.e. the only vertex in the smaller part which dominates the other part).
By Theorem~\ref{thm:Allen+}, 
\begin{itemize}
\item[(1)] {\it The class of bipartite graphs containing no 
one-sided copy of $M^*(p)$ is at most factorial}. 
\end{itemize}

The rest of the proof will follow from a series of claims.

\begin{itemize}
\item[(2)]{ \it $A_2=B_2=\emptyset$.} Suppose this is not true,
then as $G$ is connected there must be a vertex $x \in A_2 \cup B_2$ with a neighbour 
$y\in A_1\cup B_1$. Without loss of generality assume that $x \in A_2$, 
then $x,y,$ a neighbour and a non-neighbour of $y$ in $A_0$, and any $p$ vertices in $B_0$ 
induce an $M(p)$ in $G$.

\item[(3)]{ \it No vertex in $A_1''$ has a neighbour in $B_1$.} Indeed, if a vertex 
$x\in A_1''$ is adjacent to a vertex $y\in B_1$, then $x,y$ together with $p$ non-neighbours of $x$ in 
$B_0$, a neighbour and a non-neighbour of $y$ in $A_0$ induce an $M(p)$ in $G$.

\item[(4)]{ \it No vertex in $B_1''$ has a neighbour in $A_1$} by analogy with (3).

\item[(5)]{ \it The subgraph of $G$ induced by $A_1'$ and $B_1'$ is $M^*(p)$-free}.
Assume $G[A_1'\cup B_1']$ contains an induced $M^*(p)$ and let, without loss of generality, the $p+1$ vertices of 
this graph belong to $A_1'$. Each of this $p+1$ vertices has at most $p-1$ non-neighbours in 
$B_0$ and since $B_0$ contains at least $p^2$ vertices we conclude that there must be a vertex 
in $B_0$ adjacent to each of the $p+1$ vertices of the copy of $M^*(p)$. But then together (that 
vertex and the copy of $M^*(p)$) induce an $M(p)$ in $G$.  

\item[(6)] {\it The graphs $G[A_0 \cup B_1]$ and $G[B_0 \cup A_1]$ do not
contain a one-sided copy of $M^*(p)$}. Indeed, if, say, $G[A_0 \cup B_1]$ contains
a one-sided copy of $M^*(p)$ with $p+1$ vertices in $A_0$, then this copy together with
any vertex in $B_0$ induce an $M(p)$ in $G$. 
\end{itemize}

This structure obtained for graphs in the class of $M(p)$-free bipartite graphs 
containing a $K_{p^2,p^2}$ implies that such graphs can be covered by finitely many graphs from a finite union of classes 
with at most factorial speed of growth.
By Lemma~\ref{lem:covering} we conclude that the class of $M(p)$-free bipartite graphs of chordality at most $k$ 
is at most factorial for any values of $k$ and $p$.
\end{proof}

\subsubsection{$N(p)$-free bipartite graphs of bounded chordality}

By $N(p)$ we denote the graph $L(1,p)+O_{1,0}$, i.e. the graph obtained from $L(1,p)$ by adding an isolated vertex 
to the smaller part of graph.  Figure~\ref{fig:N(p)}
represents the graph $N(3)$. 

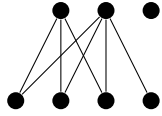
\begin{figure}[ht]
\begin{center}
	\begin{tikzpicture}
  		[scale=.6,auto=left]

		\node[vertex] (x0) at (-1,0)  { };
		\node[vertex] (x1) at (0,0)   { };
		\node[vertex] (x2) at (1,0)   { };
		\node[vertex] (x3) at (2,0)   { };
		\node[vertex] (x4) at (0,2)   { };
		\node[vertex] (x5) at (1,2)   { };
		\node[vertex] (x6) at (2,2)   { };

		\foreach \from/\to in {x1/x4,x1/x5,x2/x4,x2/x5,x5/x3,x0/x4,x0/x5}
    	\draw (\from) -- (\to);
	\end{tikzpicture}
\end{center}
\caption{Graph $N(3)$}
\label{fig:N(p)}
\end{figure}

\begin{theorem}\label{thm:N}
For any natural $k$ and $p$, the class of $N(p)$-free bipartite graphs of chordality at most $k$ 
is at most factorial.
\end{theorem}

\begin{proof}
Let $G$ be an $N(p)$-free bipartite graph of chordality at most $k$. If $G$ contains no $K_{p^2,p^2}$,
it admits an implicit representation by Lemma~\ref{lem:KSS-free}. Therefore, we assume
that $G$ contains a $K_{p^2,p^2}$. Moreover, by Theorem~\ref{thm:modular} we may assume that $G$ is prime 
and hence is connected.
We denote the two parts in the bipartition of $G$ by $A$ and $B$ and split them into 
$A_0,A_1',A_1'',A_2$ and $B_0,B_1',B_1'',B_2$ as in Theorem~\ref{thm:Qp}.

Let $N^*(p)$ denote the subgraph of $N(p)$ obtained by deleting the vertex of degree $p+1$
(i.e. the only vertex in the smaller part which dominates the other part).
By Theorem~\ref{thm:Allen+}, 
\begin{itemize}
\item[(1)] {\it The class of bipartite graphs containing no 
one-sided copy of $N^*(p)$ is at most factorial}. 
\end{itemize}

The rest of the proof will follow from a series of claims.

\begin{itemize}
\item[(2)]{ \it The subgraph of $G$ induced by $A_0$ and $B_1\cup B_2$ contains no one-sided copy of $N^*(p)$.}

\noindent
Assume, by contradiction, that this subgraph contains a copy of $N^*(p)$ with the larger part belonging to $A_0$.
Then this copy together with any vertex of $B_0$ induce an $N(p)$. 

\item[(3)]{ \it Every vertex in $A_1$ is adjacent to every vertex in $B_1'' \cup B_2$}.
 
\noindent
To prove this, assume a vertex $x \in A_1$ has a non-neighbour $y \in B_1'' \cup B_2$. 
By definition of $B_1''$ and $B_2$, vertex $y$ has at least $p$ non-neighbours in $A_0$, 
while $x$ has a neighbour and a non-neighbour in $B_0$. But then $x, y$, a neighbour and a non-neighbour of $x$ in $B_0$ 
and any $p$ non-neighbours of $y$ in $A_0$ induce an $N(p)$. 

\item[(4)]{ \it Every vertex in $B_1$ is adjacent to every vertex in $A_1'' \cup A_2$} 
by analogy with (3).

\item[(5)]{ \it The subgraph of $G$ induced by $A_2$ and $B_2$ contains no one-sided copy of $N^*(p)$.}

\noindent
To show this, we first observe that if $A_2\cup B_2$ is not empty, then $A_1\cup B_1$ is not empty, since otherwise the graph $G$ is disconnected.
Therefore, if $A_2\cup B_2$ is not empty, we may consider a vertex $x\in A_1$ (without loss of generality). 
Then the subgraph $G[A_2\cup B_2]$ contains no copy of $N^*(p)$ with the larger part belonging to $B_2$,
since otherwise this copy together with vertex $x$ induce an $N(p)$.

\item[(6)]{ \it The subgraph of $G$ induced by $A_1'$ and $B_1'$ is $N^*(p)$-free}.

\noindent
Assume $G[A_1'\cup B_1']$ contains an induced $N^*(p)$ and let the $p+1$ vertices of 
this graph belong to $A_1'$. Each of this $p+1$ vertices has at most $p-1$ non-neighbours in 
$B_0$ and since $B_0$ contains at least $p^2$ vertices we conclude that there must be a vertex 
in $B_0$ adjacent to each of the $p+1$ vertices of the copy of $N^*(p)$. But then together (that 
vertex and the copy of $N^*(p)$) induce an $N(p)$ in $G$.  
\end{itemize}

This structure obtained for graphs in the class of $N(p)$-free bipartite graphs 
containing a $K_{p^2,p^2}$ implies that such graphs can be covered by finitely many graphs from a finite union of classes with at most factorial 
speed of growth. By Lemma~\ref{lem:covering} we conclude that the class of $N(p)$-free bipartite graphs of chordality at most $k$ 
is at most factorial for any values of $k$ and $p$.
\end{proof}

\medskip
According to Remark~\ref{rem:1}, the result obtained in Theorem~\ref{thm:N} is, in a sense, best possible.

\subsubsection{$\cal A$-free bipartite graphs of bounded chordality}

By $\cal A$ we denote the graph represented in Figure~\ref{fig:A}. 

\begin{figure}[ht]
\begin{center} 
\begin{picture}(120,100)
\put(40,10){\circle*{7}}
\put(40,50){\circle*{7}}
\put(40,90){\circle*{7}}
\put(85,10){\circle*{7}}
\put(85,50){\circle*{7}}
\put(85,90){\circle*{7}}
\put(40,10){\line(0,1){40}}
\put(85,10){\line(0,1){40}}
\put(40,50){\line(0,1){40}}
\put(85,50){\line(0,1){40}}
\put(40,50){\line(1,0){45}}
\put(40,90){\line(1,0){45}}
\end{picture}
\caption{The graph  $\cal A$}
\label{fig:A}
\end{center} 
\end{figure}

\begin{theorem}
For each natural $k$, the class of $\cal A$-free bipartite graphs of chordality at most $k$ 
is at most factorial.
\end{theorem}

\begin{proof}
Let $G$ be an $\cal A$-free bipartite graph of chordality at most $k$. If $G$ contains no $C_4$,
it admits an implicit representation by Lemma~\ref{lem:KSS-free}. Therefore, we assume
that $G$ contains a $C_4$. Moreover, by Theorem~\ref{thm:modular} we may assume that $G$ is prime.

We extend the $C_4$ contained in $G$ to a maximal (with respect to set inclusion) complete bipartite graph $H$ with parts $A$ and $B$.
Observe that $|A|\ge 2$ and $|B|\ge 2$, since $H$ contains a $C_4$. We denote by $C$ the set of 
neighbours of $B$ outside $A$ (i.e. the set of vertices outside $A$ each of which has at least one neighbour in $B$) 
and by $D$ the set of neighbours of $A$ outside $B$. Notice that 
\begin{itemize}
\item[(1)] {\it $C$ and $D$ are non-empty}, since otherwise $B$ or $A$ is a non-trivial module, contradicting the primality of $G$; 
\item[(2)] {\it each vertex of $C$ has a non-neighbour in $B$ and each vertex of $D$ has a non-neighbour in $A$} due to the maximality of $H$.
\end{itemize} 
We also claim that 
\begin{itemize}
\item[(3)] {\it $C\cup D$ induces a complete bipartite graph}. Indeed, assume there are two non-adjacent vertices $c\in C$ and $d\in D$.
Consider a neighbour $b_1$ and a non-neighbour $b_2$ of $c$ in $B$, and a neighbour $a_1$ and a non-neighbour $a_2$ of $d$ in $A$.
Then the six vertices $a_1,a_2,b_1,b_2,c,d$ induce an $\cal A$ in $G$, a contradiction. 
\item[(4)] $V(G)=A\cup B\cup C\cup D$. To show this, assume there is a vertex $x\not \in A\cup B\cup C\cup D$. Without loss of generality
we may assume that $x$ is adjacent to a vertex $c\in C$ (since $G$ is prime and hence is connected). Let $d$ be any vertex of $D$,
$b$ any neighbour of $c$ in $B$, and $a_1,a_2$ a neighbour and a non-neighbour of $d$ in $A$. Then  the six vertices $a_1,a_2,b,c,d,x$ 
induce an $\cal A$ in $G$, a contradiction. 
\item[(5)] {\it every vertex of $D$ has at most one non-neighbour in $A$}. Assume, by contradiction, that a vertex $d\in D$ has two non-neighbours
$a_1,a_2$ in $A$. Since $G$ is prime, there must exist a vertex distinguishing $a_1$ and $a_2$ (otherwise $\{a_1,a_2\}$ is a non-trivial module).
Let $d'$ be such a vertex. Clearly, $d'$ belongs to $D$. Finally, consider any vertex $c\in C$ and any of its neighbours $b\in B$. 
Then   the six vertices $a_1,a_2,b,c,d,d'$ induce an $\cal A$ in $G$, a contradiction.
\item[(6)] {\it every vertex of $A$ has at most one non-neighbour in $D$}. Assume, to the contrary, that a vertex $a\in A$ has two non-neighbours
$d_1,d_2$ in $D$. Then, by (3) and (5), $a$ is the only non-neighbour of $d_1$ and $d_2$. But then $\{d_1,d_2\}$ is a non-trivial 
module, contradicting the primality of $G$. 
\end{itemize} 
Claims (5) and (6) show that the bipartite complement of $G[A\cup D]$ is a graph of vertex degree at most 1.
Moreover, in this graph at most one vertex of $A$ and at most one vertex of $D$ have degree less than 1  (since $G$ is prime).
By symmetry, the bipartite complement of $G[B\cup C]$ is a graph of degree at most 1 with at most one vertex of degree 0 in each part.
Therefore, $G$ can be covered by at most 4 graphs each of which admits an implicit representation (by Lemma~\ref{lem:degree}). 
As a result, by Lemma~\ref{lem:implicit}, $G$ admits an implicit representation and hence the class under consideration is at most factorial. 
\end{proof}


\subsection{$P_7$-free bipartite graphs}
\label{sec:P7}

As we mentioned earlier, the class of $P_7$-free bipartite graphs
is the only class defined by a single forbidden induced subgraph 
for which the membership in the factorial layer is an open question.
To better understand this stubborn case, in this section we systematically 
study subclasses of $P_7$-free bipartite graphs obtained by forbidding one 
more graph. 

First, we observe that all the results obtained in the previous section 
are applicable to $P_7$-free bipartite graphs, because these graphs are 
$(C_8,C_9,\ldots)$-free.

Next, we list a number of subclasses of $P_7$-free bipartite graphs for which 
the membership in the factorial layer is either known or easily follows from some known
results. In particular, from Theorem~\ref{thm:Allen} we known that $(P_7,G)$-free bipartite 
graphs constitute a factorial class for any graph $G\ne P_7$ such that neither $G$ nor its 
bipartite complement contains a cycle. Also, two more results follow readily from Lemma~\ref{lem:forest}.

\begin{cor}\label{cor:C6}
The classes of $(P_7,C_6)$-free and $(P_7,3K_2)$-free bipartite graphs are factorial.
\end{cor}

\begin{proof}
Both classes contain $2K_2$-free bipartite graphs, which proves a lower bound.
To show an upper bound, we observe that the class of $(P_7,C_6)$-free bipartite graphs coincides with $P_7$-free chordal
bipartite graphs and hence is at most factorial by Lemma~\ref{lem:forest}. Also, the class of $(P_7,3K_2)$-free bipartite 
graphs coincides with the bipartite complements of $(P_7,C_6)$-free bipartite graphs and hence is at most factorial too.  
\end{proof}

\subsubsection{$(P_7,S_{p,p})$-free bipartite graphs}
By $S_{p,q}$ we denote a double star, i.e. the graph obtained from two stars $K_{1,p}$ and $K_{1,q}$
by connecting their central vertices with an edge. 

\begin{theorem}\label{lem:double-star}
For any $p$, the class of $(P_7,S_{p,p})$-free bipartite graphs 
admits an implicit representation and hence is at most factorial.
\end{theorem}    

\begin{proof}
Let $G$ be a $(P_7,S_{p,p})$-free bipartite graph. By Lemma~\ref{lem:implicit} we assume that $G$ is connected.
If $G$ does not contain $K_{p,p}$ as an induced subgraph, then $G$ can be 
described implicitly by Lemma~\ref{lem:KSS-free}. So suppose $G$ contains a $K_{p,p}$.
If $G=K_{p,p}$, then obviously it can be described implicitly. Therefore, we assume
that the set of neighbours of the $K_{p,p}$ is non-empty. We denote this set by $A$ and apply 
Lemma~\ref{lem:partial-implicit} (keeping in mind remarks of Section~\ref{sec:remarks}).

First, we show that $G[A]$ can be represented implicitly. To this end, 
for each edge $(u,v)$ of the $K_{p,p}$, we denote by $H_{uv}$ the subgraph 
of $A$ induced by the neighbours of $u$ and the neighbours of $v$.
This subgraph must be $O_{p,p}$-free, since otherwise any copy of this subgraph
together with $u$ and $v$ would induce an $S_{p,p}$. Clearly, every pair of
vertices of $A$ (from different parts of the bipartition) belongs to at least 
one subgraph $H_{uv}$ and hence the set of all these subgraphs gives a covering 
of $G[A]$. Also each vertex of $A$ is covered by at most $p^2$ subgraphs in 
the covering, because $p^2$ is the total number of such subgraphs. 
Finally, we observe that each $H_{uv}$ admits an implicit representation, 
because each of them is the bipartite complement of a $(P_7,K_{p,p})$-free 
bipartite graph, which admits such a representation by Lemma~\ref{lem:KSS-free}.
Therefore, by Lemma~\ref{lem:implicit}, $G[A]$ admits an implicit representation.

Second, we show that each vertex of $A$ has at most $2p-1$ neighbours outside of this set. 
Indeed, each vertex of $A$ has at most $p$ neighbours in the $K_{p,p}$. Now assume
a vertex $u\in A$ has at least $p$ neighbours outside of $K_{p,p}\cup A$. 
Observe that $u$ must also have a neighbour $v$ in the $K_{p,p}$. But then 
vertices $u$ and $v$ together with the neighbours of $v$ in the $K_{p,p}$
and the $p$ neighbours of $u$ outside of $K_{p,p}\cup A$ induce an $S_{p,p}$.
This contradiction shows that each vertex $u$ of $A$ has at most $p-1$ neighbours outside of 
$K_{p,p}\cup A$ and hence at most $2p-1$ neighbours outside of $A$. 

Combining the two facts above, we conclude by Lemma~\ref{lem:partial-implicit}
that $G$ can be represented implicitly. 
\end{proof}

\medskip
To conclude this section, we observe that the result of Theorem~\ref{lem:double-star}
cannot be extended to graphs of bounded chordality, because 

\begin{remark}\label{rem:2}
For any $p\ge 2$, the class of $(P_8,S_{p,p})$-free bipartite graphs is superfactorial. 
\end{remark}

This conclusion follows from the fact that for any $p\ge 2$, the class of $(P_8,S_{p,p})$-free bipartite graphs
contains all $\widetilde{C}_4$-free graphs.  

\subsubsection{($P_7, K_{p,p}+O_{0,p}$)-free bipartite graphs}


%
%



By $B(p,q)$ we denote the bipartite Ramsey number, i.e. the minimum number such that every bipartite graph with
at least $B(p,q)$ vertices in each of the parts contains either $K_{p,q}$ or $O_{p,q}$  as an induced subgraph.



\begin{lemma} \label{lm:KssO0s}
	For every $p \in \mathbb{N}$, any $(K_{p,p} + O_{0,p})$-free bipartite graph $G=(A,B,E)$ is either 
$K_{t,t}$-free or $O_{t,t}$-free, where $t = B(p,p)+p-1$.
\end{lemma}
\begin{proof}
	Suppose, by contradiction, that $G$ contains $K_{t,t}$ and $O_{t,t}$ as induced subgraphs. Denote by
	$A_K \subseteq A$, $B_{K} \subseteq B$ the parts of the $K_{t,t}$ and by $A_O \subseteq A$, $B_O \subseteq B$ the
	parts of the $O_{t,t}$.

	Obviously, either $A_K \cap A_O = \emptyset$ or $B_K \cap B_O = \emptyset$. Without loss of generality, assume that
	$A_K \cap A_O = \emptyset$. If $|B_K \cap B_O| \geq p$ then any $p$ vertices from $A_K$, any $p$ vertices from 
	$B_K \cap B_O$ and any $p$ vertices from $A_0$ induce forbidden $K_{p,p} + O_{0,p}$. If $|B_K \cap B_O| < p$,
	then $|B_K \setminus B_O| \geq B(p,p)$ and $|B_O \setminus B_K| \geq B(p,p) > p$. Therefore, 
	$G[B_K \setminus B_O \cup A_O]$ contains either $K_{p,p}$ or an induced $O_{p,p}$. In the former case
	$G[B_K \setminus B_O \cup A_O \cup B_O \setminus B_K]$ contains $K_{p,p}+O_{0,p}$ as an induced subgraph and
	in the latter case $G[A_K \cup B_K \setminus B_O \cup A_O]$ contains the forbidden induced subgraph.
	This contradiction proves the lemma.
\end{proof}

\begin{theorem}\label{thm:KO}
	For every $p \in \mathbb{N}$, the class of $(P_7,K_{p,p}+O_{0,p})$-free bipartite graphs is at most factorial.
\end{theorem}
\begin{proof}
	From Lemma~\ref{lm:KssO0s} it follows that the class of $(P_7,K_{p,p}+O_{0,p})$-free bipartite graphs is contained 
in the union  $Free(P_7, K_{t,t}) \cup Free(P_7, O_{t,t})$,
	where $t = B(p,p)+p-1$. Since $P_7=\widetilde{P}_7$, from Lemma~\ref{lem:KSS-free} it follows that 
both classes in the union are at most factorial. Therefore,
	the class of $(P_7,K_{p,p}+O_{0,p})$-free bipartite graphs also is at most factorial.
\end{proof}

\medskip
By analogy with Remark~\ref{rem:2}, we conclude that 

\begin{remark}\label{rem:3}
For any $p\ge 2$, the class of $(P_8,K_{p,p}+O_{0,p})$-free bipartite graphs is superfactorial. 
\end{remark}

Therefore, Theorem~\ref{thm:KO} cannot be extended to graphs of bounded chordality.

\subsubsection{($P_7, K_{1,2} + 2K_2$)-free bipartite graphs}

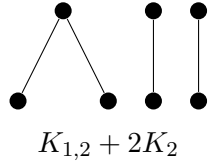
\begin{center}
	\begin{tikzpicture}
  		[scale=.6,auto=left]

%

		\node[vertex] (x1) at (0,0)   { };
		\node[vertex] (x2) at (2,0)   { };
		\node[vertex] (x3) at (3,0)   { };
		\node[vertex] (x4) at (4,0)   { };
		\node[vertex] (x5) at (1,2)   { };
		\node[vertex] (x6) at (3,2)   { };
		\node[vertex] (x7) at (4,2)   { };

		\foreach \from/\to in {x1/x5, x2/x5, x3/x6, x4/x7}
    	\draw (\from) -- (\to);
		\coordinate [label=center:$K_{1,2}+2K_2$] (K122K2) at (2,-1);
	\end{tikzpicture}
\end{center}

\begin{lemma}
	Let $G=(V,E)$ be a $(P_7, K_{1,2}+2K_2)$-free bipartite graph. 
	Then $G$ either is $3K_2$-free or has two
	vertices $a,b$ such that $|N(a) \Delta N(b)| = 2$.
\end{lemma}

\begin{proof}
	We will show that if $G$ contains $3K_2$ as an induced subgraph, then it has two
	vertices $a,b$ such that $|N(a) \Delta N(b)| = 2$.
	
	Suppose that a set of vertices $M = \{ x_1, y_1, x_2, y_2, x_3, y_3 \} \subseteq V$ induces
	a $3K_2$ such that $(x_i, y_i) \in E$ for $i=1,2,3$. If some vertex $v \notin M$ has a neighbour in
	$A = \{x_1, x_2, x_3\}$, then it has at least two neighbours in this set, because otherwise 
	$v, x_1, y_1, x_2, y_2, x_3, y_3$ would induce a forbidden $K_{1,2}+2K_2$.
	
	If two vertices $v$ and $w$ have exactly two neighbours in $A$, then $N(v) \cap A = N(w) \cap A$. Indeed,
	if say $v$ is adjacent to $x_1, x_2$ and $w$ is adjacent to $x_2, x_3$, then $y_1, x_1, v, x_2, w, x_3, y_3$
	induce a $P_7$, which is impossible.
	
	Thus each vertex outside $M$ either has no neighbours in $A$, or is adjacent to all vertices of 
	$A$, or is adjacent to exactly two particular vertices, say $x_1, x_2$. This implies that 
	$N(x_1) \Delta N(x_2) = \{y_1, y_2\}$.
\end{proof}

This lemma together with Corollaries~\ref{cor:Delta},~\ref{cor:C6} and remarks of Section~\ref{sec:remarks} imply the following conclusion.

\begin{theorem}
The class of $(P_7, K_{1,2}+2K_2)$-free bipartite graphs is factorial.
\end{theorem}

\subsubsection{($P_7, P_5 + K_2$)-free bipartite graphs}

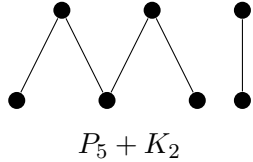
\begin{center}
	\begin{tikzpicture}
  		[scale=.6,auto=left]

%
%

		\node[vertex] (x1) at (0,0)   { };
		\node[vertex] (x2) at (2,0)   { };
		\node[vertex] (x3) at (4,0)   { };
		\node[vertex] (x4) at (5,0)   { };
		\node[vertex] (x5) at (1,2)   { };
		\node[vertex] (x6) at (3,2)   { };
		\node[vertex] (x7) at (5,2)   { };

		\foreach \from/\to in {x1/x5, x2/x5, x3/x6, x4/x7, x6/x2}
    	\draw (\from) -- (\to);
		\coordinate [label=center:$P_5 + K_2$] (P5K2) at (2.5,-1);
	\end{tikzpicture}
\end{center}

\begin{lemma}\label{lm:2K2_free}
	Every $2K_2$-free bipartite graph with at least three vertices has two
	vertices $x$ and $y$ which are in the same part and $|N(x) \Delta N(y)| \leq 1$.
\end{lemma}

\begin{proof}
	Let $G=(V_1,V_2,E)$ be a $2K_2$-free bipartite graph. It is known that the 
	vertices in each of the parts can be ordered linearly with respect to inclusion 
	of their neighbourhoods. Suppose that $|V_1|=n_1 \geq |V_2|$ and let 
	$V_1 = \{ x_1, x_2, \ldots, x_{n_1} \}$ such that 
	$N(x_1) \supseteq N(x_2) \ldots \supseteq N(x_{n_1})$. If $|N(x_1) \Delta N(x_2)| > 1$
	then there are at least two vertices in $N(x_1) \setminus N(x_2)$. All these vertices
	have the same neighbourhood $\{x_1\}$ and hence any two of them meet the condition of 
	the statement.	
\end{proof}

\begin{lemma}
	Let $G=(V_1,V_2,E)$ be a $\{ P_7, P_5 + K_2 \}$-free bipartite graph.
	Then $G$ either is $3K_2$-free or  has two 
	vertices $a,b$ such that $|N(a) \Delta N(b)| \leq 4$.
\end{lemma}

\begin{proof}
	Suppose that $G$ contains $3K_2$ as an induced subgraph. We will show that $G$ has two vertices $a,b$
	such that $|N(a) \Delta N(b)| \leq 4$. 

	Let $\{(x_1,y_1), (x_2,y_2), \ldots, (x_s,y_s)\} \subseteq E, s \geq 3$  be a maximum induced matching 
	in $G$ such	that $M_1 = \{x_1, \ldots, x_s\} \subseteq V_1$ and 
	$M_2 = \{y_1, \ldots, y_s\} \subseteq V_2$.
	
	Every vertex $v$ outside $M_1$($M_2$) either has no neighbours in $M_2$($M_1$) or it is connected to
	all vertices from $M_2$($M_1$) or it has exactly one neighbour in $M_2$($M_1$). Indeed, if say 
	$v \in V_1$ is adjacent to $y_i, y_j \in M_2$ and is not adjacent to 
	$y_k \in M_2$, $i,j,k \in \{1, \ldots, s\}$, then $x_i,y_i,v,y_j,x_j,x_k,y_k$ induce a 
	forbidden $P_5 + K_2$.
	
	According to this observation, we denote by $A_1$($A_2$) the set of vertices which are adjacent to 
	every vertex in $M_2$($M_1$) and by $B_1$($B_2$) the set of vertices with exactly one neighbor in 
	$M_2$ ($M_1$). Let $C_1 = V_1 \setminus (A_1 \cup M_1 \cup B_1)$ and 
	$C_2 = V_2 \setminus (A_2 \cup M_2 \cup B_2)$. Let $X_i = N_{B_2}(x_i)$ and $Y_i = N_{B_1}(y_i)$
	for $i=1,\ldots,s$. From the definition it follows that $B_1$($B_2$) is the union of
	disjoint sets $Y_i$ ($X_i$), $i=1,\ldots,s$.
	
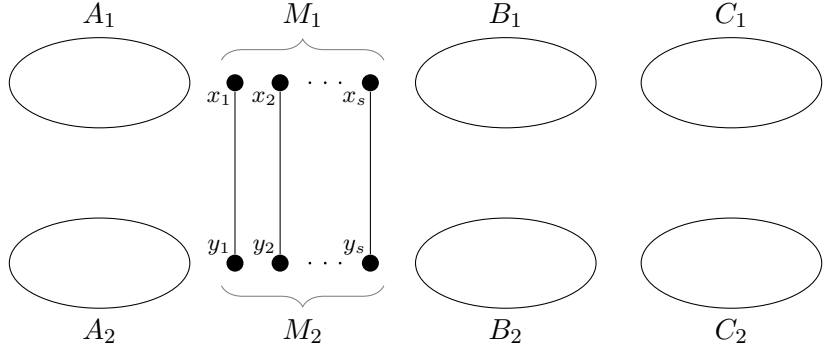
	\begin{center}
		\begin{tikzpicture}
	  		[scale=.6,auto=left]
			\draw (0,4) ellipse [x radius = 2, y radius = 1]; 
			\coordinate [label=center:$A_1$] (A1) at (0,5.5);
			
			\draw (0,0) ellipse [x radius = 2, y radius = 1];
			\coordinate [label=center:$A_2$] (A2) at (0,-1.5);			

			\node[vertex] (x1) at (3,4) { };
			\coordinate [label=center:\footnotesize{$x_1$}] (x_1) at (2.65,3.65);
			\node[vertex] (x2) at (4,4) { };
			\coordinate [label=center:\footnotesize{$x_2$}] (x_2) at (3.65,3.65);
			\node[point] (px1) at (4.66, 4) {};
			\node[point] (px2) at (5, 4) {};
                                \node[point] (px3) at (5.33, 4) {};
			\node[vertex] (xs) at (6,4) { };
			\coordinate [label=center:\footnotesize{$x_s$}] (x_s) at (5.65,3.65);
			\draw [gray, decorate, decoration={brace, amplitude=8}] (2.7,4.5)  --  (6.3,4.5);
			\coordinate [label=center:$M_1$] (M1) at (4.5,5.5);
			\node[vertex] (y1) at (3,0)   { };
			\coordinate [label=center:\footnotesize{$y_1$}] (y_1) at (2.65,0.35);
			\node[vertex] (y2) at (4,0)   { };
			\coordinate [label=center:\footnotesize{$y_2$}] (y_2) at (3.65,0.35);
			\node[point] (py1) at (4.66, 0) {};
			\node[point] (py2) at (5, 0) {};
                                \node[point] (py3) at (5.33, 0) {};
			\node[vertex] (ys) at (6,0)   { };
			\coordinate [label=center:\footnotesize{$y_s$}] (y_s) at (5.65,0.35);
			\draw [gray, decorate, decoration={brace, amplitude=8}] (6.3,-0.5)  --  (2.7,-0.5);
			\coordinate [label=center:$M_2$] (M2) at (4.5,-1.5);

			\foreach \from/\to in {x1/y1,x2/y2,xs/ys}
	    		\draw (\from) -- (\to);

			\draw (9,4) ellipse [x radius = 2, y radius = 1]; 
			\coordinate [label=center:$B_1$] (B1) at (9,5.5);
			
			\draw (9,0) ellipse [x radius = 2, y radius = 1];
			\coordinate [label=center:$B_2$] (B2) at (9,-1.5);				

			\draw (14,4) ellipse [x radius = 2, y radius = 1]; 
			\coordinate [label=center:$C_1$] (C1) at (14,5.5);
			
			\draw (14,0) ellipse [x radius = 2, y radius = 1];
			\coordinate [label=center:$C_2$] (C2) at (14,-1.5);			
		\end{tikzpicture}
	\end{center}
	
	\begin{enumerate}
		\item[(1)] \textit{Let $i,j \in \{1, \ldots, s\}$ and $i \neq j$. Then no vertex in $X_i$ has a neighbour in $Y_j$.}
		
		\noindent
		Assume, by contradiction, that $v \in X_i$ is adjacent to $u \in Y_j$. But then
		vertices $y_i$, $x_i$, $v$, $u$, $y_j$, $x_k$, $y_k$ with  $k\in \{1, \ldots, s\}$, $k \neq i$ and $k \neq j$
		induce $P_5+K_2$.
		
		\item[(2)] \textit{Every vertex from $A_1$($A_2$) is adjacent to every vertex from $B_2$($B_1$).}
		
		\noindent
		For the sake of definiteness, suppose that $v \in A_1$ is not adjacent to $u \in X_i \subseteq B_2$,
		$i \in \{1, \ldots, s\}$. But then $x_j,y_j,v,y_k,x_k,x_i,u$, where $j \neq i$, $k \neq i$ and
		$j,k \in \{1, \ldots, s\}$, would induce a forbidden $P_5+K_2$.
		
		\item[(3)] \textit{Every vertex from $C_1(C_2)$ has neighbours in at most one of the sets 
		$X_i$($Y_i$), $i=1,\ldots,s$}.
		
		\noindent
		Assume, by contradiction, that $v \in C_1$ is adjacent to $u_1 \in X_i$ and to $u_2 \in X_j$,
		$i \neq j$, but then $y_i,x_i,u_1,v,u_2,x_j,y_j$ would induce a forbidden $P_7$.
	\end{enumerate}
	
	From (3) it follows that $C_1$($C_2$) is a union of disjoint sets $R, R_1, \ldots, R_s$ 
	($Q, Q_1, \ldots, Q_s$), where $R$($Q$) is the set of vertices which have no neighbours in
	$B_2$($B_1$) and $R_i$($Q_i$), $i \in \{1, \ldots, s\}$, is the set of vertices which
	have at least one neighbour in $X_i$($Y_i$) and have no neighbour in $X_j$($Y_j$), $j \neq i$
	and $j \in \{1, \ldots, s\}$.
	
	\begin{enumerate}
		\item[(4)] \textit{Every vertex from $A_1$($A_2$) is adjacent to every vertex in 
		$C_2 \setminus Q$ ($C_1 \setminus R$).}
		
		\noindent
		Indeed, if say $v \in A_1$ is not adjacent to $u \in Q_i \subseteq C_2 \setminus Q$ and
		$w$ is a neighbour of $u$ in $Y_i$, then $x_j, y_j, v, y_k, x_k, w, u$, where $j \neq i$,
		$k \neq i$ and $i,j,k \in \{1, \ldots, s\}$, would induce a forbidden $P_5+K_2$.
	\end{enumerate}
	
	Note that for every $i \in \{1, \ldots, s\}$, $G[X_i \cup C_2 \cup Y_i \cup C_1]$ is $2K_2$-free,
	because otherwise $M$ would not be maximum. For the same reason there are no edges between $C_1$
	and $C_2$.
	
	We may assume that there exists $i \in \{1, \ldots, s\}$ such that $|X_i| \geq 2$ and $|Y_i| \geq 2$,
	otherwise there are at least $\left\lceil s/2 \right\rceil$ vertices in one of the parts $M_1, M_2$
	which have at most one neighbour in $B_1 \cup B_2$ and hence for any two of these vertices $a,b$ we
	have $|N(a) \Delta N(b)| \leq 4$. Consider graph $G[X_i \cup Q_i \cup Y_i \cup R_i]$. As it is
	$2K_2$-free then by Lemma \ref{lm:2K2_free} it has two vertices $v,u$ which are in the same part, say
	$Y_i \cup R_i$, such that $|N_{X_i \cup Q_i}(v) \Delta N_{X_i \cup Q_i}(u)| \leq 1$. Note that 
	$|N_{M_2}(v) \Delta N_{M_2}(u)| \leq 1$. Also from (2) and (4) it follows that 
	$|N_{A_2}(v) \Delta N_{A_2}(u)| = 0$. Together with (1) it implies that 
	$|N(v) \Delta N(u)| \leq 2$.
\end{proof}

This lemma together with Corollaries~\ref{cor:Delta},~\ref{cor:C6} and remarks of Section~\ref{sec:remarks} imply 
the following conclusion.

\begin{theorem}
The class of $(P_7, P_5 +K_2)$-free bipartite graphs is factorial.
\end{theorem}

\subsubsection{($P_7, C_4+K_2$)-free bipartite graphs}

	
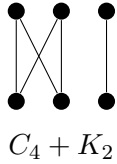
\begin{center}
		\begin{tikzpicture}
	  		[scale=.6,auto=left]
%
%
%

			\node[vertex] (x1) at (0,0)   { };
			\node[vertex] (x2) at (1,0)   { };
			\node[vertex] (x3) at (2,0)   { };
			\node[vertex] (x4) at (2,2)   { };
			\node[vertex] (x5) at (0,2)   { };
			\node[vertex] (x6) at (1,2)   { };
	
			\foreach \from/\to in {x1/x5,x1/x6, x2/x5,x2/x6, x3/x4}
	    	\draw (\from) -- (\to);
			\coordinate [label=center:$C_4+K_2$] (C4K2) at (1,-1);
		\end{tikzpicture}
	\end{center}

\begin{lemma} \label{lm:2K2_C4_free}
	Let $H=(A,B,E)$ be a $(2K_2,C_4)$-free bipartite graph. Then in each part at most
	one vertex has degree more then 1 and all vertices with degree 1 have the same neighborhood.
\end{lemma}
\begin{proof}
	We prove the statement for the part $A$. For the part $B$ the same arguments are true.
	Let $x,y$ be some vertices from $A$. Then
	$N(x) \subseteq N(y)$ or $N(y) \subseteq N(x)$, otherwise forbidden $2K_2$ would arise.
	From this in particular follows that  all vertices with degree 1 have the same neighborhood.
	Also, for the same reason, there is at most one vertex with degree more then 1 in $A$, otherwise 
	forbidden $C_4$ would arise.
\end{proof}

\begin{lemma}
Let $G=(V_1,V_2,E)$ be a $(P_7, C_4+K_2)$-free bipartite graph.
Then $G$ either  is $3K_2$-free or has two vertices $a,b$ such that $|N(a) \Delta N(b)| \leq 8$.	
\end{lemma}
\begin{proof}
We	suppose that $G$ contains $3K_2$ as an induced subgraph and show that it has  two vertices $a,b$
	such that $|N(a) \Delta N(b)| \leq 8$. 

	Let $\{(x_1,y_1), (x_2,y_2), \ldots, (x_s,y_s)\} \subseteq E, s \geq 3$  be a maximum induced matching in $G$ such
	that $M_1 = \{x_1, \ldots, x_s\} \subseteq V_1$ and $M_2 = \{y_1, \ldots, y_s\} \subseteq V_2$. Denote by
	 $A_1$ ($A_2$) the set of vertices which are adjacent to every vertex in $M_2$
	($M_1$) and by $B_1$ ($B_2$) all other vertices which have neighbors in $M_2$ ($M_1$). 
	Let $C_1 = V_1 \setminus (A_1 \cup M_1 \cup B_1)$ and $C_2 = V_2 \setminus (A_2 \cup M_2 \cup B_2)$. Then

	\begin{enumerate}
		\item[(1)] \textit{For any two vertices $v,u \in B_1$ ($B_2$), either $N_{M_2}(v) \cap N_{M_2}(u) = \emptyset$ or
		$N_{M_2}(v) \subseteq N_{M_2}(u)$ or $N_{M_2}(u) \subseteq N_{M_2}(v)$ 
		(either $N_{M_1}(v) \cap N_{M_1}(u) = \emptyset$ or $N_{M_1}(v) \subseteq N_{M_1}(u)$ or 
		$N_{M_1}(u) \subseteq N_{M_1}(v)$)}.

\noindent		
        Suppose, by contradiction, that there are $v,u$ in $B_1$ and $y_i,y_j,y_k$ in $M_2$ such that
		$(v,y_i) \in E$,  $(v,y_k) \in E$, $(v,y_j) \notin E$, $(u,y_j) \in E$,  $(u,y_k) \in E$ and $(u,y_i) \notin E$. But
		then $x_i, y_i, v, y_k, u, y_j, x_j$ induce a forbidden $P_7$.
			
		\item[(2)] \textit{For any two vertices $y_i, y_j \in M_2$, $|N_{B_1}(y_i) \cap N_{B_1}(y_j)| \leq 1$ 
		(for any two vertices $x_i, x_j  \in M_1$, $|N_{B_2}(x_i) \cap N_{B_2}(x_j)| \leq 1$)}.
		
		\noindent
		Suppose, by contradiction, that there are $y_i, y_j$ in $M_2$ such that there are two different vertices
		$v$ and $u$ in $N_{B_1}(y_i) \cap N_{B_1}(y_j)$. From (1) it follows that $N_{M_2}(v) \subseteq N_{M_2}(u)$
		or $N_{M_2}(u) \subseteq N_{M_2}(v)$. Without loss of generality, assume that $N_{M_2}(v) \subseteq N_{M_2}(u)$.
		By definition of $B_1$, there is a vertex $y_k$ in $M_2$ which is not adjacent to $u$ and hence is not adjacent 
		to $v$. Therefore $x_k,y_k,v,u,y_i,y_j$ induce a forbidden $C_4+K_2$.

		\item[(3)] \textit{For any vertex $y_i \in M_2$, at most one vertex from $N_{B_1}(y_i)$ has a neighbor in 
		$M_2$ different from $y_i$}.
		
		\noindent
		Suppose that there are two vertices $v$ and $u$ in $N_{B_1}(y_i)$ such that there are $y_j$ and $y_k$ in $M_2$
		different from $y_i$ and $(v,y_j) \in E$ and $(u,y_k) \in E$. We have  $(v,y_k) \notin E$ and $(u,y_j) \notin E$, otherwise
		$y_i$ and $y_k$ or $y_i$ and $y_j$ have more then one common neighbor in $B_1$, which contradicts (2). But then
		$x_j,y_j,v,y_i,u,y_k,x_k$ induce a forbidden $P_7$.

		\item[(4)] \textit{Let $v \in B_1$ be adjacent to exactly one vertex from $M_2$, say $y_i$. Then $v$ is adjacent to
		at most two vertices from $B_2$}.
		
		\noindent
		Suppose, by contradiction, that $v$ has three neighbors $c,d,e$ in $B_2$. Assume $c$ is adjacent to $x_i$. Remember 
		that $c$ must also have a non-neighbour $x_j\in M_1$. But then vertices $c,x_i,y_i,v$ together with $x_j,y_j$
induce a forbidden $C_4+K_2$. This contradiction shows that $c$ is not adjacent to $x_i$. Similarly, $d$ and $e$ are 
not adjacent to $x_i$. 
		
		Let $L_1 = M_1 \setminus \{x_i\}$. We claim that the graph $G[L_1 \cup \{c,d,e\}]$ is $(2K_2,C_4)$-free.
		Indeed, if, say, $c,x_j,d,x_k$ induce a $2K_2$, then $y_j,x_j,c,v,d,x_k,y_k$ induce a forbidden 
		$P_7$, and if, say, $c,x_j,d,x_k$ induce a $C_4$, then $c,x_j,d,x_k$ together with $x_i,y_i$  induce a forbidden $C_4+K_2$.

		From Lemma \ref{lm:2K2_C4_free} it follows that there are two vertices in $\{c,d,e\}$ say $c,d$ which have the same
		neighborhood in $L_1$ consisting of exactly one vertex, say $x_k$. But then $v,c,d,x_k,x_j,y_j$, where 
		$x_j \in L_1$ and $x_j \neq x_k$, induce a forbidden $C_4+K_2$.

		\item[(5)] \textit{Let $v \in B_1$ be adjacent to exactly one vertex from $M_2$, say $y_i$. Then $v$ is adjacent to
		all but at most one vertex from $A_2$}.
		
		\noindent
		Suppose, by contradiction, that $u,w \in A_2$ are not adjacent to $v$. But then vertices $y_i$, $v$, $u$, $w$, $x_j$, $x_k$, where $j$ and $k$
		are different from $i$, induce a forbidden $C_4+K_2$.

		\item[(6)] \textit{Let $R_i \subseteq B_1$ be the set of vertices which are adjacent only to $y_i$ in $M_2$. Then
		$G[R_i \cup C_2]$ is a $(2K_2,C_4)$-free graph}.
		
		\noindent
		If $G[R_i \cup C_2]$ contains an induced $2K_2$, then this contradicts the maximality of the matching
		$\{(x_1,y_1), (x_2,y_2), \ldots, (x_s,y_s)\}$. If $G[R_i \cup C_2]$ contains a $C_4$ then $(x_k,y_k)$, 
		where $k \neq i$, together with the $C_4$ constitute a forbidden $C_4+K_2$.

		\item[(7)] \textit{ $G$ has two vertices $a,b$ such that $|N(a) \Delta N(b)| \leq 8$}.
		
		\noindent
		If $M_2$ has no vertices with more than 3 neighbors in $B_1$, then we can take as $a$ and $b$ any two
		vertices from $M_2$. Otherwise, if there is vertex $y_i$ in $M_2$ which has at least $4$ neighbors in $B_1$, then
		by (3) $N_{B_1}(y_i)$ contains three vertices, say $c,d,e$, which are adjacent only to $y_i$ in $M_2$. From (6)
		and Lemma \ref{lm:2K2_C4_free} it follows that for two of these vertices, say $c,d$, 
		$|N_{C_2}(c) \Delta N_{C_2}(d)| \leq 1$. From (5) it follows that $|N_{A_2}(c) \Delta N_{A_2}(d)| \leq 2$ and from (4)
		it follows that $|N_{B_2}(c) \Delta N_{B_2}(d)| \leq 4$. Therefore $|N(c) \Delta N(d)| \leq 7$ and we can take $c$ and $d$
		as $a$ and $b$.
	\end{enumerate}
\end{proof}

This lemma together with Corollaries~\ref{cor:Delta},~\ref{cor:C6}  and remarks of Section~\ref{sec:remarks} imply
the following conclusion.

\begin{theorem}
The class of $(P_7, C_4 +K_2)$-free bipartite graphs is factorial.
\end{theorem}

\subsubsection{($P_7, domino$)-free bipartite graphs}

\begin{center}
	\begin{tikzpicture}
  		[scale=.6,auto=left]

		\node[vertex] (x1) at (0,0)   { };
		\node[vertex] (x2) at (1,0)   { };
		\node[vertex] (x3) at (2,0)   { };
		\node[vertex] (x4) at (0,2)   { };
		\node[vertex] (x5) at (1,2)   { };
		\node[vertex] (x6) at (2,2)   { };

		\foreach \from/\to in {x1/x4,x1/x5,x2/x4,x2/x5,x5/x3,x6/x3,x2/x6}
    	\draw (\from) -- (\to);
		\coordinate [label=center:$domino$] (C402) at (1,-1);
	\end{tikzpicture}
\end{center}

\begin{theorem}\label{thm:domino}
The class of ($P_7, domino$)-free bipartite graphs is factorial.
\end{theorem}

\begin{proof}
This class contains all $2K_2$-free bipartite graphs (one of the three minimal factorial classes of bipartite graphs) and 
hence is at least factorial. Now let us show an upper bound. 

Since the class of ($P_7, C_4$)-free bipartite graphs is at most factorial (Lemma~\ref{lem:KSS-free}), we consider 
a  ($P_7, domino$)-free bipartite graph containing a $C_4$. We extend this $C_4$ to a maximal 
biclique in $G$ with parts denoted by $A$ and $B$ of size at least 2. Then we define 
$C=N(B) \backslash A$, 
$D=N(A) \backslash B$, 
$E=N(D) \backslash (A\cup C)$, 
$F=N(C) \backslash (B \cup D)$,
 $I=N(F) \backslash (A \cup C \cup E)$, 
 $J=N(E) \backslash (B \cup D \cup F)$. 
Now we prove a series of claims. 

\begin{itemize}
\item[(1)]
{\it The set $C \cup D$ is independent}. Suppose, by contradiction, there is 
an edge $cd$ with $c\in C$ and $d \in D$. By definition $c$ must have a neighbour $b_1$ in $B$. 
Also, as $G[A \cup B ]$ is a maximal biclique, $c$ has a non-neighbour $b_0$ in $B$. 
Similarly $d$ has a neighbour $a_1$ and a non-neighbour $a_0$ in $A$. 
But then $G$ contains a $domino$ induced by $c, d, a_0, a_1, b_0, b_1$. 

\item[(2)]
{\it The subgraph induced by $A \cup D$ does not contain one-sided copy of a $P_5$ with 3 vertices in $A$, and hence it is $(P_6, C_6)$-free}.
Assume $G[A \cup D]$ contains a one-sided copy of a $P_5$ with 3 vertices in $A$. Then this copy together with any vertex $b\in  B$ 
induces a $domino$, a contradiction.

\item[(3)]  By symmetry, the subgraph induced by  $B \cup C$ is $(P_6, C_6)$-free.

\item[(4)]
{\it The set $F \cup E$ is independent}. Suppose, by contradiction, there is an edge 
$(e,f)$ with $e \in E$ and $f \in F$. Then pick a neighbour $c \in C$ of $f$ and a neighbour $d \in D$ of $e$,
which must exist by definition of $E$ and $F$. Let $b \in B$ be a non-neighbour of $c$, and let 
$a_1 \in A$ and $a_2 \in A$ be a neighbour and a non-neighbour of $d$. Then vertices $a_1, b, a_2, d, e, f, c$ induce a $P_7$, a contradiction. 

\item[(5)]
{\it The subgraph induced by $F \cup C$ is $C_6$-free}. 
Suppose, by contradiction, there is a $C_6$ induced by $c_1,f_1,c_2,f_2,c_3,f_3$ with $c_1, c_2, c_3 \in C$ and $f_1, f_2, f_3 \in F$. 
If there is a vertex $b \in B$ adjacent to all $c_1, c_2, c_3$, then $G$ contains a $domino$ induced by 
$b, c_1, f_1, c_2, f_2, c_3$. Also, if there is a vertex $b \in B$ adjacent to exactly one of $c_1, c_2, c_3$, say $c_1$, 
then together with any vertex $a \in A$ we have a $P_7$ induced by $a,b,c_1,f_1,c_2,f_2,c_3$. 
If there is a vertex $b_1 \in B$ not adjacent to any of $c_1, c_2, c_3$, then take a vertex $b_2 \in B$ adjacent to 2 of them, 
say $c_1$ and $c_3$, and together with any vertex $a \in A$ form a $P_7$ induced by $b_1,a,b_2,c_1,f_1,c_2,f_2$. Therefore, 
each vertex of $B$ is adjacent to exactly two of $c_1, c_2, c_3$ and since all vertices in $C$ have a non-neighbour in $B$, 
each pair must appear. So, pick vertex $b_1$ adjacent to $c_1$ and $c_2$, pick vertex $b_2$ adjacent to $c_2$ and $c_3$
and pick vertex $b_3$ adjacent to $c_3$ and $c_1$. Now $G[\{c_1,b_1,c_2,b_2,c_3,b_3\}]$ is a $C_6$, contradicting our Claim (3). 
This contradiction shows that the graph $G[F \cup C]$ is $C_6$-free.

\item[(6)]
By symmetry, the  subgraph induced by $E \cup D$ is $C_6$-free.

\item[(7)]
{\it The set $I \cup J$ is independent}. If not, assume $(i,j)$ is an edge with $i \in I$ and $j \in J$. 
Then take a neighbour $e \in E$ of $j$, a neighbour $d \in D$ of $e$, a neighbour $a_1$ and a non-neighbour $a_2$ of $d$ in $A$ 
and an arbitrary vertex $b$ in $B$. Then $i,j,e,d,a_1,b,a_2$ induce a $P_7$, a contradiction. 

\item[(8)]
{\it The subgraph induced by $J \cup E$ is $C_6$-free}. Suppose, by contradiction, 
there is a $C_6$ induced by $e_1,j_1,e_2,j_2,e_3,j_3$ with $e_1, e_2, e_3 \in E$ and $j_1, j_2, j_3 \in J$. 
Now if there is a vertex $d \in D$ joined to all $e_1, e_2, e_3$, then $G$ contains a $domino$ induced by 
$d, e_1, j_1, e_2, j_2, e_3$. Otherwise, there is a vertex $d \in D$ having a neighbour and a non-neighbour 
in the set $\{e_1, e_2, e_3 \}$, say $d$ is non-adjacent to $e_1$ and adjacent to $e_2$. 
Then pick a neighbour $a_1 \in A$ of $d$ and a non-neighbour $a_2 \in A$ of $d$. 
Pick an arbitrary $b \in B$. Then $e_1,j_1,e_2,d,a_1,b,a_2$ induce a $P_7$ in $G$, a contradiction.  

\item[(9)]
By symmetry, the subgraph induced by $I \cup F$ is $C_6$-free. 

\item[(10)]
{\it If $G$ is connected, then $V(G)=A \cup B \cup C \cup D \cup E \cup F \cup I \cup J$}. 
Suppose that $N(J) \backslash E \neq \emptyset$ and take an edge $(j,k)$ with $j \in J$ and $k \in N(J) \backslash E$. 
Then take a neighbour $e \in E$ of $j$, a neighbour $d \in D$ of $e$, a neighbour $a_1$ and a non-neighbour $a_2$ of $d$ in $A$ 
and an arbitrary vertex $b$ in $B$. Then $i,j,e,d,a_1,b,a_2$ induce a $P_7$. This contradiction shows that $N(J) \backslash E = \emptyset$. 
By symmetry, $N(I) \backslash F = \emptyset$. Hence the claim. 
\end{itemize}

This series of claims shows that every connected ($P_7, domino$)-free bipartite graph can be 
covered by finitely many graphs each coming from a class which is at most factorial. By Lemma~\ref{lem:covering}
this implies that the class of ($P_7, domino$)-free bipartite graphs is at most factorial. 
\end{proof}

\subsubsection{($P_7, K_{3,3}$-$e$)-free bipartite graphs}

The graph $K_{3,3}$-$e$ is obtained from $K_{3,3}$ by deleting an edge. 

\begin{theorem}
The class of ($P_7, K_{3,3}$-$e$)-free bipartite graphs is factorial.
\end{theorem}

\begin{proof}
The class of ($P_7, K_{3,3}$-$e$)-free bipartite graphs contains all graphs of degree at most one 
(one of the three minimal factorial classes of bipartite graphs) and 
hence is at least factorial. Now we show an upper bound. In the proof we follow the structure and
notation of the proof of Theorem~\ref{thm:domino}. In particular, we assume that 
a connected ($P_7, K_{3,3}$-e)-free bipartite graph $G$ contains a $K_{4,4}$ and denote by $A,B,C,D,E,F,I,J$
the subsets defined in the proof of Theorem~\ref{thm:domino}. From this theorem we know that these subsets partition the vertex set of $G$ 
and that $I\cup J$ is an independent set, since otherwise an induced $P_7$ arises. 
Also, by definition, the subgraph of $G$ induced by $A\cup B$ is complete bipartite
with at least 4 vertices in each part. Now we derive a number of claims as follows.

\begin{itemize}
\item[(1)] {\it Every vertex outside of $A\cup B$ has at most one neighbour in $A\cup B$}. To show this, 
consider a vertex $x\not\in A\cup B$ which has at least two neighbours in $A$. 
By definition, $G[A\cup B]$ is a maximal biclique and hence $x$ also has a non-neighbour in $A$.
But then a non-neighbour and two neighbours of $x$ in $A$ together with $x$ and any two vertices of $B$ 
induce a $K_{3,3}$-$e$.
\item[(2)] {\it The subgraph induced by $C\cup D\cup F\cup J$ is $domino$-free}. To prove this, assume, by contradiction, that this subgraph contains a $domino$ induced by vertices $x_1,x_2,x_3\in C$ and $y_1,y_2,y_3$
with $x_2$ and $y_2$ being the vertices of degree 3 in the induced $domino$. By definition, $x_1$ has a neighbour
$z$ in $B$.  Also, since vertices $x_1,x_2$ and $x_3$ have collectively at most 3 neighbours
in $B$ and the size of $B$ is at least 4, there must exist a vertex $b\in B$ adjacent to none of $x_1,x_2,x_3$.
For the same reason, one can find a vertex $a$ in $A$ which is 
adjacent to none of $y_1,y_2,y_3$. Now if $z$ is not adjacent to $x_2$, then vertices $b,a,z,x_1,y_1,x_2,y_3$
induce a $P_7$ in $G$, and if $z$ is not adjacent to $x_3$, then vertices $b,a,z,x_1,y_2,x_3,y_3$
induce a $P_7$ in $G$, and if $z$ is adjacent to both $x_2$ and $x_3$, then vertices $z,x_1,x_2,x_3,y_2,y_3$
induce a $K_{3,3}$-$e$ in $G$. A contradiction in all possible cases proves the claim. 
\item[(3)] {\it The subgraph induced by $E\cup F\cup J$ is $domino$-free}. This can be proved by analogy with (2).
We assume, by contradiction, that this subgraph contains a $domino$ induced by vertices $x_1,x_2,x_3\in E$ and $y_1,y_2,y_3$
and consider a neighbour $z$ of $x_1$ in $D$, a neighbour $a$ of $z$ in $A$ and an arbitrary vertex $b$ in $B$.
Then the very same arguments as in (2) lead to a contradiction.    
\end{itemize}
By symmetry we conclude that the subgraphs of $G$ induced by $D\cup C\cup E\cup I$ and by $F\cup E\cup I$ are  $domino$-free.
Therefore, $G$ can be covered by finitely many $domino$-free graphs. Together with Lemma~\ref{lem:covering} and 
Theorem~\ref{thm:domino} this completes the proof.  
\end{proof}

\subsubsection{Bipartite complements of $P_7$-free bipartite graphs}

Since the bipartite complement of $P_7$ is again $P_7$, from the preceding sections 
we derive the following conclusion.

\begin{theorem}
Let $H$ be the bipartite complement of any of the following graphs: $Q(p)$, $L(s,p)+O_{0,1}$, $M(p)$, $N(p)$, $\cal A$, 
$S_{p,p}$, $K_{p,p}+O_{p,p}$, $K_{1,2}+2K_2$, $P_5+K_2$, $C_4+K_2$, $domino$, $K_{3,3}$-$e$.
The class of $(P_7,H)$-free bipartite graphs is at most factorial. 
\end{theorem}

This theorem together with the results of the preceding sections imply, in particular,
that for any graph $H$ with at most 6 vertices the class of $(P_7,H)$-free bipartite graphs 
is at most factorial. 

\section{Concluding remarks and open problems}

To simplify the study of the family of factorial classes, in this paper we introduce
several tools and apply them to reveal new hereditary classes with the
factorial speed of growth. 
However, the problem of finding a global structural characterization of the factorial classes is still far from being solved.
We do not even know such a characterization for the classes of bipartite graphs  defined by a single forbidden induced subgraph. The speed of such classes was studied in \cite{Allen} and the class of $P_7$-free bipartite graphs is the unique class in this family for which the membership in the factorial layer is still an open question. Answering this question is a challenging research problem.
To better understand the structure of $P_7$-free bipartite graphs, in the present
paper we consider subclasses of this class defined by one additional forbidden
induced subgraph and prove, in particular, that for every graph $G$ with
at most 6 vertices the class of $(P_7,G)$-free bipartite graphs is at most factorial.

Also, some of the introduced tools can be used to obtain a conclusion on the
existance of an implicit representation of a given class. For many of the new
revealed factorial classes we show that they admit an implicit representation.
Though, the implicit graph conjecture is still an open challenging problem.

\end{document}